\documentclass[10pt]{amsart}

\usepackage{amsthm}
\usepackage{amsmath}
\usepackage{graphicx}
\usepackage[margin=1.2in]{geometry}
\usepackage{amssymb,amsfonts}
\usepackage[all,arc]{xy}
\usepackage{enumerate}
\usepackage{mathrsfs}
\usepackage{bbm}
\usepackage{color}
\usepackage{verbatim}

\DeclareMathOperator{\diver}{div}

\newcommand{\rr}{\ensuremath{\mathbb{R}}}

\newcommand{\ep}{\ensuremath{\varepsilon}}
\newcommand{\bdry}{\ensuremath{\partial}}

\newcommand{\balpha}{\ensuremath{\bar{\alpha}}}

\newtheorem{thm}{Theorem}[section]

\newtheorem{prop}[thm]{Proposition}
\newtheorem{lem}[thm]{Lemma}

\theoremstyle{definition}

\newtheorem{defn}[thm]{Definition}

\theoremstyle{remark}
\newtheorem{rem}[thm]{Remark}

\numberwithin{equation}{section}

\begin{document}
\title{Uniform convergence for the incompressible limit of a tumor growth model}
\date{\today}
\author{Inwon Kim}
\address{Department of Mathematics, UCLA, Los Angeles, CA, 90095}
\thanks{Inwon Kim is supported by NSF grant DMS-1566578.}
\email{ikim@math.ucla.edu}
\author{Olga Turanova}
\address{Department of Mathematics, UCLA, Los Angeles, CA, 90095}
\thanks{Olga Turanova is supported by NSF grant DMS-1502253.}
\email{turanova@math.ucla.edu}
\begin{abstract}
We study a  model introduced by Perthame and Vauchelet \cite{PV} that describes the growth of a tumor governed by Brinkman's Law, which takes into account friction between the tumor cells. We adopt the viscosity solution approach to establish an optimal uniform convergence result of the tumor density as well as the pressure in the incompressible limit. The system lacks standard maximum principle, and thus modification of the usual approach is necessary.

\end{abstract}
\maketitle

\section{Introduction}

We study the following model, which was introduced by Perthame and Vauchelet in \cite{PV}. It describes the growth of tumors at the cellular level by providing a law relating the cell density, pressure, and cell multiplication. The tumor cell density $n_k: \rr^n\times [0,\infty) \to \rr$ satisfies,
\begin{equation}
\label{eq:density}
\begin{cases}
\partial_t n_k - \diver(n_k DW_k)=n_k G(p_k),\\
-\nu \Delta W_k+W_k = p_k,
\end{cases}
\end{equation}
where the pressure $p_k$ is given by,
\[
p_k=\frac{k}{k-1}(n_k)^{k-1}.
\]
Here $\nu$ is a positive constant and $G$ is a given function that describes the effect that the pressure has on the growth of the tumor. We assume $G$ satisfies,
\begin{equation}
\label{G basic}
G\in C^1(\rr),\text{ }  G'(\cdot)\leq -\balpha<0, \text{ and }G(P_M)=0 \text{ for some } P_M>0 \text{ and }\bar{\alpha}>0.
\end{equation}

The  main results of \cite{PV} concern the limit as $k\rightarrow\infty$, or the so-called incompressible limit, of (\ref{eq:density}). This connects  (\ref{eq:density}) to a system that involves a moving front. If the parameter $\nu$ were zero (in other words, if the tumor were governed by Darcy's Law), then the system (\ref{eq:density}) would become,
\[
\partial_t n_k - \diver(n_k Dp_k)=n_k G(p_k).
\]
This model for tumor growth has been widely studied, and we refer the reader to the introduction of \cite{PV} for a variety of references, both about modeling and rigorous mathematical analysis. In particular, in  \cite{PQV}, Perthame, Quir\`os and V\'azquez find that  the incompressible limit of the above equation is the Hele-Shaw problem with a forcing term.  Kim and Pozar \cite{KP} used viscosity solution methods to improve the  result in \cite{PQV}.  The model that we study, (\ref{eq:density}) with $\nu>0$, has been proposed as a better description of tumor growth. Here, the  tumor  is governed by Brinkman's Law, which takes into account the friction between the tumor cells, and not just of the tumor with its environment.  These modeling issues are discussed in, for example, \cite{ZWC, BCGRS}.

Of particular interest in the asymptotic limit is the limiting pressure, which represents the incompressibility condition. In the  inviscid model ($\nu=0$), the limiting pressure solves a Hele-Shaw type problem and is continuous as long as the pressure zone is reasonably regular \cite{PQV}. However, as illustrated in \cite{PV}, in the viscous model that we study here the limiting pressure is strictly positive on the boundary of its support, and thus is  discontinuous. This is an interesting contrast to the inviscid model.

Our goal in this paper is to obtain   pointwise convergence results  in the framework of viscosity solutions theory, improving the $L^1$ convergence  obtained in \cite{PV}. Due to the discontinuity of the limiting pressure, the optimal pointwise convergence result one expects is uniform convergence  away from the pressure boundary. This is precisely what we obtain. In addition, knowing that the pressure converges uniformly then allows us to  improve the convergence of the $W_k$ as well (see Theorem~\ref{main result} below).

We point out that the system (1.1) does not enjoy the comparison principle -- in fact, it is strongly coupled -- and thus one needs to modify the existing theory in the analysis.  To achieve this we follow the approach in \cite{Kim}, where we rely on the fact that one component of the  system can be considered almost fixed due to its strong convergence: in our case that turns out to be the $W_k$, though their convergence is still weaker than what is available in \cite{Kim}.

\subsection*{Heuristics}
Let us briefly recall the formal derivation of the limiting system given in \cite{PV} to illustrate additional challenges and main ingredients of our analysis in more detail. We denote the limit of $(p_k, n_k, W_k)$ by $(p_\infty, n_\infty, W_\infty)$. 
Perhaps the easiest equation to guess is the one for $W_{\infty}$:
 \begin{equation}\label{elliptic}
 -\nu\Delta W_{\infty} + W_{\infty} = p_{\infty}.
 \end{equation}
 Next we expect that $p_{\infty}$ is either zero or satisfies $p_{\infty} -\nu G(p_{\infty}) = W_{\infty}$. This is because we can write the $n_k$ equation in terms of $p_k$ as 
 \begin{equation}\label{transport_k_pressure}
 \partial_t p_k - Dp_k\cdot DW_k = (k-1)\nu^{-1} p_k( W_k-(Id-\nu G)(p_k)),
 \end{equation}
 which then translates $p_{\infty}$ as a singular limit of reaction-diffusion equations. Thus it is reasonable to think that $p_\infty$ will take value either zero or $(Id-\nu G)^{-1}(W_\infty)$. In other words, we expect to have $p_\infty = (Id-\nu G)^{-1}(W_\infty)\chi_{\Omega_t}$ for some region $\Omega_t$. The question now is to characterize $\Omega_t$.

We recall that there is a third component here, namely $n_k$.   
  Manipulating the equation for $n_k$ and then using the equation that $W_k$ satisfies yields,
 \begin{align}
 \label{eq:nk}
 \partial_tn_k - Dn_k\cdot DW_k &=n_k(\Delta W_k+G(p_k))=\frac{n_k}{\nu}(W_k-p_k+\nu G(p_k)).
 \end{align}
The region $\Omega_t$ is where the $p_k$ converge to the positive value $(Id-\nu G)^{-1} (W_{\infty})$, so by definition we know that the $n_k$ converge to $1$ there. 
When the $p_k$ converge to 0 (in other words, on $\Omega_t^c$) we expect the $n_k$ to converge to zero if initially this is the case (see the discussion in the outline below). 
Notice that in both situations, the right-hand side of the previous equation is zero. Thus we expect $n_\infty$ to equal $\chi_{\Omega_t}$ and solve,
 \begin{equation}\label{transport}
 \partial_t n_{\infty} - Dn_{\infty} \cdot DW_{\infty} = 0,
 \end{equation}
 yielding the normal velocity law for the set $\Omega_t$. Thus, we expect the triple $(p_\infty, n_\infty, W_\infty)$ to solve the system,
\begin{equation}
\label{system}
\begin{cases}
-\nu\Delta W_{\infty}  + W_{\infty} = p_{\infty},\\ 
p_{\infty} = (Id-\nu G)^{-1}(W_{\infty}) \chi_{\{n_{\infty}>0\}}, \quad n_{\infty} = \chi_{\{n_{\infty}>0\}},\\
\partial_t n_{\infty} - Dn_{\infty} \cdot DW_{\infty} = 0.\\

\end{cases}
\end{equation}

  The above heuristics are indeed true when the limiting density $n_{\infty}$ is initially a patch. Then it follows from the transport equation (\ref{transport}) that $n_{\infty}$ satisfies that $n_{\infty}$ is always zero or one at later times. In general the limiting system \eqref{system} is invalid with the presence of the region $\{0<n_{\infty}<1\}$, due to the interaction of the two convergence regions as $k\to\infty$. In the inviscid model ($\nu=0$), this was studied by \cite{KP} and Mellet, Perthame and Quir\`os \cite{MPQ}. In our situation the normal velocity of the pressure zone in this general setting remains open.
\subsection*{Initial data}
Let us now state the conditions on the limiting initial data with the notation $H:= (Id-\nu G)^{-1}$, as given in \cite{PV}. We assume,
  \begin{equation}\label{limit_initial}
 n_{\infty} = \chi_{\Omega_0},  \quad p_{\infty}^0 = H(W_{\infty}^0)\chi_{\Omega_0},   \quad -\nu\Delta W_{\infty} + W_{\infty} =  p_{\infty}^0,
 \end{equation}
where $\Omega_0\subset \rr^n$ is a compact set with measure zero boundary. The last two equations are, as mentioned in \cite{PV}, to avoid initial layers in the limit system.  As for the approximating system, we impose 
 \begin{equation}\label{initial}
\liminf{}_* p_k(x,0) >0 \hbox{ on } \Omega_0\hbox{ and } \liminf_{k\rightarrow\infty} dist\left(\{x|p_k(x,0)>0\}, (\bar{\Omega_0})^c\right)>0,
 \end{equation}
where $dist$ is the usual distance function.  
The assumptions (\ref{initial}) that we make on the initial data are very similar to those in \cite{PV}, but they are neither more nor less general. We discuss this further below.

\subsection*{Main result}
Now we are ready to state our main result.

 \begin{thm}
\label{main result}
Let $\Omega_0$ be a compact set in $\rr^n$ and let $n_k$ and $W_k$ solve (\ref{eq:density}) with initial data satisfying (\ref{initial}). Then, along a subsequence: 
 \begin{itemize}
 \item[(a)] the $W_k$ converge strongly to $W_\infty$ in $L^{\infty}((0,T), W^{2,p}_{loc}(\rr^n))$,
 \item[(b)] the $p_k$ converge locally uniformly to $p_{\infty}$ on $(\rr^n\times (0,\infty))\setminus \bdry\{n_{\infty}>0\}$,
 \item[(c)] the $n_k$ converge locally uniformly to $n_\infty$ on $(\rr^n\times (0,\infty))\setminus \bdry\{n_{\infty}>0\}$,
\end{itemize}
where $(p_{\infty}, n_{\infty}, W_{\infty})$ solve (\ref{system}) with initial data (\ref{limit_initial}). Moreover, $\bdry\{n_{\infty}>0\}$ has measure zero.
 \end{thm}

As stated in the theorem, the limiting density  has its support evolving by the geometric flow (\ref{transport}). Since our goal here  is to obtain the converge of the density and pressure in a strong sense -- namely, locally uniformly -- we therefore need to employ a sufficiently strong notion of solution for (\ref{transport}). For this reason we consider viscosity solutions to (\ref{transport}). This  allows us to use barrier arguments with smooth test functions, as well as stability properties, to yield the (locally) uniform convergence results that we desire. Since a priori estimates only yield $DW_\infty$ to be integrable in time and log-Lipschitz in space,  (\ref{transport}) is not covered by standard viscosity solutions theory. Thus, a key part of our work is to define a notion of viscosity solution for (\ref{transport}); establish basic results such as stability, existence and a comparison theorem; and describe the unique viscosity solution of (\ref{transport}) in terms of the associated flow map. In fact, in the proof of the main result we identify $n_\infty$ with the function given by (\ref{formula}) of Theorem \ref{thm:discontinuous}, with $V=-DW_\infty$.

Before we discuss the main ingredients of the proof in more detail, some remarks on the theorem are in order.
\subsubsection*{Size of $ \bdry\{n_\infty>0\}$}
Theorem \ref{main result} tells us the limiting behavior of the  pressure and density everywhere except on  $\bdry\{n_\infty>0\}$. In Lemma \ref{lem:Gammat} we show that $ \bdry\{n_\infty>0\}$ has Lebesgue measure zero for all times. 
If $DW_\infty$ were Lipschitz in space, then the flow generated by $DW_\infty$ would also be Lipschitz, and, for example, the Hausdorff dimension of $\bdry\{n_\infty>0\}$ would be preserved under the flow. This is not quite the case for us: see Section \ref{sec:further} for more discussion.

\subsubsection*{Relationship of our work and \cite{PV}} Our results strengthen those of \cite{PV} in several ways. First, we obtain locally uniform convergence of $p_k$ and $n_k$, improving the $L^1_{loc}$ convergence in \cite{PV}. Second, we  characterize $n_\infty$ as an indicator function. This confirms what was suggested   in the numerical examples in \cite[Section 3]{PV}, but was not proven there. Third, as a consequence of the stronger convergence of the $p_k$, we improve the convergence of the $W_k$ from $L^1((0,T), W^{1,q}_{loc}(\rr^n))$ to $L^{\infty}((0,T), W^{2,p}_{loc}(\rr^n))$.

We now discuss in detail the relationship between our assumption on the initial data (\ref{initial}) and the analogous assumption in \cite{PV}. Indeed, \cite{PV} assumes that $p_k(x,0)$ converge almost everywhere to $H(W_\infty)$ on $\Omega_0$, and are identically zero on $\Omega_0^c$. These assumptions of \cite{PV} do not imply that (\ref{initial}) holds, because convergence in the almost everywhere sense is weaker than what is needed for (\ref{initial}) to hold. On the other hand, our assumptions do not imply that those in \cite{PV} hold. We do not need to assume convergence of the $p_k(x,0)$ to $H(W_\infty)$ on $\Omega_0$; for us it is enough only to assume that the $p_k$ are uniformly positive there. In addition, we do not require all of the $p_k(x,0)$ to have the same zero set; we simply require a convergence of the zero sets. Thus, our assumptions are neither stronger nor weaker than those in \cite{PV}.

\subsubsection*{The initial time}We discuss the behavior of the $p_k$ near the initial time. First, we point out a difference between the assumptions on the initial data for the limiting system (\ref{limit_initial}) and for the system at the $k$-level (\ref{initial}) -- the second condition in  (\ref{limit_initial}) states $p_{\infty}^0 = H(W_{\infty}^0)\chi_{\Omega_0}$; however, at the $k$-level we assume only $\liminf{}_* p_k(x,0) >0$ on $\Omega_0$. Despite this, we are able to establish that the $p_k$ converge locally uniformly to $p_\infty$ for $t>0$ (and off of $\bdry \{n_\infty>0\}$).

To see why this should be the case, we look at the equation that $p_k$ satisfies, (\ref{transport_k_pressure}), and explain the heuristics. We see that if $p_k$ is even a little bit positive initially, it will approach the stable root of the reaction term  for positive times as $k$ approaches infinity. This causes a possible ``jump" at time 0: indeed, it is even possible for the initial data $p_k(x,0)$ to  converge to something other than $p_{\infty}^0$, and yet for the $p_k$ to still converge to $p_\infty$ for $t>0$.

\subsection*{Main challenges and ingredients}

As mentioned above, our goal here  is to obtain the convergence of the $p_k$ in a strong sense -- namely, locally uniformly, using the viscosity solution approach. We illustrate the main ingredients and challenges of the proof below.

\subsubsection*{Viscosity solutions for the transport equation} 

A key part of our work is defining and establishing basic properties for viscosity solutions of (\ref{transport}).  We prove a comparison result, Theorem \ref{thm:comparison}, that is essential to the rest of the paper, and implies that solutions to (\ref{transport}) with continuous initial data are unique. In addition, we also establish uniqueness for solutions to (\ref{transport}) that have a characteristic function as initial data (see Theorem \ref{thm:discontinuous}). This is an interesting and subtle point -- in general, it is possible for a Hamilton Jacobi equation $u_t+H(x,t,Du)=0$ to enjoy uniqueness for continuous solutions, but not discontinuous solutions (see the counterexample of Barles, Soner and Souganidis in \cite[Proposition 4.4]{BSS} where non-uniqueness occurs for $u_t +(x-t)|Du|=0$ due to nucleation).

\subsubsection*{Literature on the transport equation} 
There is a wide literature on renormalized solutions (in the sense of DiPerna-Lions \cite{DL}) and distributional solutions to the transport equation with quite general vector fields. In particular, Ambrosio's \cite[Theorem 4.1]{Ambrosio} establishes uniqueness of distributional solutions to
\[
u_t+Du\cdot b(x,t)=0
\]
where  $b(\cdot,t)\in BV_{loc}(\rr^n)$ for almost all $t$ and satisfies $\diver(b)\in L^1((0,T),L^{\infty}_{loc}(\rr^n))$. Our vector field $DW_\infty$ satisfies these hypotheses. However, the aforementioned result concerns solutions in the distributional sense. We do not know whether or not a viscosity solution is a distributional solution, and thus cannot immediately deduce uniqueness or comparison for our situation, so we establish comparison for viscosity solutions of (\ref{transport}) directly.

\subsubsection*{Generalized set evolution}

Understanding the evolution of the set $ \bdry\{n_\infty>0\}$ is key to finding the asymptotic behavior of our system. The heuristics indicate that this region travels with normal velocity $DW_\infty$. We want a precise and direct way to describe and study such evolution. For this, we extend the definition given by Barles and Souganidis in \cite{BS} of generalized flow to velocities that are only integrable in time. Heuristically, the definition involves testing a subset of $\rr^n$ from the ``inside" or ``outside" by smoothly evolving sets. Whether a region $\Omega_t$ is a generalized flow with velocity $DW_\infty$ is closely related to whether the indicator of $\Omega_t$ is a viscosity solution of (\ref{transport}) (this is made precise in Theorem \ref{thm:soln implies flow}). 

In \cite{BS}, the authors also introduce a way of studying the development of interfaces in asymptotic limits of reaction-diffusion equations. In \cite{Kim}, such methods were used to study a system with no comparison principle. Although the methods of \cite{BS} and \cite{Kim} provided a lot of inspiration for our work, we do not use the so-called ``abstract approach" introduced in \cite{BS}, and are able to proceed  with more basic barrier arguments. This is mostly due to the fact that the equation (\ref{transport_k_pressure}) for $p_k$ is first order.

\subsubsection*{Obtaining the main result}
Once  we have introduced the notion of generalized flow, we establish:
\begin{prop}
\label{prop:flows}
Assume the hypotheses of Theorem \ref{main result}. For $t\geq 0$, define the sets $A_{k}$, $\Omega_t^1$ and $\Omega_2^t$ by,
\[
A^{k}_t=\{x|p_k(x,t)>0\}, 
\]
\[
\Omega_t^1=\{x|\liminf{}_*p_k(x,t)>0\} \text{ and }\Omega^2_t=\{x|\liminf{}_* dist(x,A^k_t)>0\}.
\]
 Then:
\begin{enumerate}
\item \label{item:superflow} $(\Omega_t^1)^{int}$ is a generalized superflow with  velocity $-DW_\infty$, and,
\item \label{item:subflow} $(\bar{\Omega_t^2})^c$ is a generalized subflow with  velocity $-DW_\infty$.
\end{enumerate} 
\end{prop}
(Here $\liminf{}_*$ and $\limsup{}^*$ are the usual weak limits. We write down the definition in Definition \ref{def:liminf} for the convenience of the reader.) We view this proposition as the heart of our paper. It captures the basic idea that the limiting behavior of (\ref{eq:density}) can be expressed by saying where the limit of the $p_k$ is zero, where  it is positive, and how these two regions evolve in time. 

We remark on the definition of $\Omega^2$. It says that the $p_k$ are eventually zero, uniformly on compact subsets of $\Omega^2$. Knowing $(x,t)\in \Omega^2$ is strictly stronger than simply  $\limsup{}^*p_k(x,t) =0$. (As an elementary example, consider the sequence of functions $f_k(x)\equiv 1/k$. This sequence satisfies $\limsup{}^*f_k=0$, and yet the $f_k$ are never eventually zero.) In addition, $\limsup{}^*p_k(x,t)=0$ does not imply that the limit of the $n_k$ is zero, but $(x,t)\in \Omega^2$ does.

In order to deduce Theorem \ref{main result} from Proposition \ref{prop:flows}, we also study the sets $\Omega^1_t$ and $\Omega^2_t$ at the initial time, and compare them to the set $\Omega_0$ appearing in the hypotheses on the initial data. Then we establish that in $\Omega_t^1$, not only are the $p_k$ uniformly positive, but that they in fact converge to $(Id-G)^{-1}(W_\infty)$. Finally, use these results, together with estimates on the size of $\bdry \{n_\infty>0\}$, to obtain the improved convergence of the $W_k$.

\subsection*{Structure of our paper}
Viscosity solutions and generalized flows are defined in Section \ref{sec:visc soln} and Section \ref{sec:gen flows}, respectively. There we state basic properties of these two notions and of their relationship to each other. We could not find these results elsewhere in our precise setting so we include their proofs. However, since this is not the main focus of our work, these proofs are presented in the appendices, which are quite long as a result. 

Section \ref{sec:prelim} is short and covers some preliminary results about (\ref{eq:density}) that we use in the rest of the paper. Then, Section \ref{sec:superflow} and Section \ref{sec:subflow} are devoted to the proofs of items (\ref{item:superflow}) and (\ref{item:subflow}), respectively, of Proposition \ref{prop:flows}. We study the limiting behavior at the initial time in Section \ref{sec:time 0}. Section \ref{sec:positive region} is devoted to studying the limit of the $p_k$ in the positive region. We put all of the ingredients together in Section \ref{sec:proof main} and establish our main result.

\section{Viscosity solutions}
\label{sec:visc soln}
We define a notion of solution  for
\begin{equation}
\label{eq:main}
u_t+Du\cdot V( x,t)=0 
\end{equation}
for $V$ uniformly bounded, log-Lipschitz in $x$ and $L^1$ in $t$. 
Our precise hypotheses are that there exist $M>0$ and $N>0$ with:
\begin{equation}
\label{bdsV}
|V(x,t)|\leq M \text{ for all }x,t;
\end{equation}
\begin{equation}
\label{hyp:logL}
|V(x,t)-V(y,t)|\leq \sigma(|x-y|)\text{ for all }x,y,t, \text{ where } \sigma(r)=Nr|\ln(r)|;
\end{equation}
and
\begin{equation}
\label{hyp:L1}
t\mapsto V(x,t) \text{ is integrable on $(0,T)$ for any $x$.}
\end{equation}

We use $USC$ and $LSC$ to denote, respectively, the classes of real-valued upper-semicontinuous and lower-semicontinuous functions on $\rr^n$. We will also employ the upper-semicontinuous and lower-semicontinuous envelopes, which we denote for a given function $u$ by $u^*$ and $u_*$, respectively.

We follow Ishii \cite{Ishii} in defining viscosity solutions for (\ref{eq:main}). We use $H$ to denote, \[
H(x,t,p)=V(x,t)\cdot p.
\]
First we define:

\begin{defn}
For any open set $Q\subset  \rr^n\times [0,T]$,  $(x_0, t_0)\in Q$, and $p_0\in \rr^n$ we define $H^+$ and $H^-$ as:
\begin{align*}
H^+(x_0,t_0, p_0) = \{(G, b) &\text{ such that } G\in C(Q\times\rr^n), b\in L^1(0,T), \\ &G(x,t,p)+b(t)\geq H(x,t,p) \text{ for all } (x,p)\in B_\delta(x_0,p_0),\\
& \text{ almost all }t\in B_\delta(t_0), \text{some }\delta>0\},
\end{align*}
and 
\begin{align*}
H^-(x_0,t_0, p_0) = \{(G, b) &\text{ such that } G\in C(Q\times\rr^n), b\in L^1(0,T), \\ &G(x,t,p)+b(t)\leq H(x,t,p) \text{ for all } (x,p)\in B_\delta(x_0,p_0),\\
& \text{ almost all }t\in B_\delta(t_0), \text{some }\delta>0\}.
\end{align*}
\end{defn}

\begin{defn}
Let $Q$ be an open subset of $\rr^n\times [0,T]$.
\begin{enumerate}
\item $u\in USC$ is called a \emph{viscosity subsolution} in $Q$ if 
\[
\phi_t(x_0,t_0)+G(x_0,t_0,D\phi(x_0,t_0))\leq 0
\]
holds for any $\phi\in C^1(Q)$, $(x_0,t_0)\in Q$, and $(G,b)\in H^-(x_0,t_0,D\phi(x_0,t_0))$ such that 
\[
(x,t)\mapsto u(x,t)+\int_0^tb(s)\, ds-\phi(x,t)
\]
has a local maximum at $(x_0,t_0)$.
\item $u\in LSC$ is called a \emph{viscosity supersolution} in $Q$ if 
\[
\phi_t(x_0,t_0)+G(x_0,t_0,D\phi(x_0,t_0))\geq 0
\]
holds for any $\phi\in C^1(Q)$, $(x_0,t_0)\in Q$, and $(G,b)\in H^+(x_0,t_0,D\phi(x_0,t_0))$ such that 
\[
(x,t)\mapsto u(x,t)+\int_0^tb(s)\, ds-\phi(x,t)
\]
has a local minimum at $(x_0,t_0)$.
\item For $Q'\subset \rr^n$, $u$ is called a \emph{viscosity solution} on $Q'\times (0,T)$  \emph{with initial data} $u_0$ on  if $u_*$ is a supersolution, $u^*$ is a subsolution,  $u_*(x,0)\geq (u_0)_*(x)$ for all $x$, and $u^*(x,0)\leq (u_0)^*$ for all $x$.
\end{enumerate}
\end{defn}
From now on, we often use ``solution" to refer to ``viscosity solution", and similarly for sub- and super- solutions. We remark that if $u_0$ and $u$ are continuous, then $u$ being a viscosity solution with initial data $u_0$ simply means that $u$ is both a sub- and a super- solution, and equals $u_0$ at time 0.

\subsection{Basic properties}
If $H$ is continuous in $t$, then a viscosity solution in this sense is also a viscosity solution in the usual sense. 
\begin{lem}
\label{lem:usualsoln}
Suppose $H$ is continuous in $t$, $u\in USC$ (resp. LSC) is a viscosity subsolution (resp. supersolution), $\phi\in C^1$, and $u-\phi$ has a local maximum (resp. minimum) at $(x_0,t_0)$. Then 
\[
\phi_t(x_0,t_0)+H(D\phi(x_0,t_0), x_0, t_0)\leq 0 \text{ resp. }\geq 0.
\]
\end{lem}
\begin{proof} This follows directly from the definition of viscosity subsolution and supersolution.
\end{proof}

An important property of classical viscosity solution is stability -- if there is a uniformly convergent sequence of viscosity solutions, then the limit is also a viscosity solution. The next proposition asserts that the notion of viscosity solutions we define also enjoys such a property.

\begin{prop}
\label{prop:stability}
Let $H$ and $H_n$ be functions on $\rr^n\times (0,T)\times \rr^n$ for all $n=1,2,...$ such that $H_n\rightarrow H$ in $L^1((0,T),C(K))$ for any compact subset $K$ of $\rr^n\times (0,T)\times \rr^n$. Let $Q\subset \rr^n\times (0,T)$. Suppose $u_n\in C(Q)$ are subsolutions (respectively, supersolutions) of 
\[
u_t+H_n(x,t,Du)=0 \text{ in } Q \text{ for all }n.
\]
If $u_n\rightarrow u$ locally uniformly on $Q$, for some $u\in C(Q)$, then $u$  is a viscosity subsolution (respectively, supersolution) of
\[
u_t+H(x,t,Du)=0 \text{ in } Q.
\]
\end{prop}
The proof of \cite[Proposition 7.1]{Ishii} carries over to this situation with no modifications.

We establish the following comparison theorem for viscosity solutions:
\begin{thm}
\label{thm:comparison}
Suppose $V$ satisfies hypotheses (\ref{bdsV}), (\ref{hyp:logL}) and (\ref{hyp:L1}). Suppose $u\in USC(\rr^n\times[0,T])$ and $v\in LSC(\rr^n\times[0,T])$ are, respectively, a sub and super solution to (\ref{eq:main}) on $\rr^n\times [0,T]$. Then,
\[
\sup_{\rr^n\times [0,T]} (u-v)\leq \sup_{\rr^n}(u(x,0)- v(x,0))^+.
\] 
\end{thm}
This result is essential to our work. The proof is essentially a doubling-variables argument. It has two parts. First, we establish that $u(x,t)-v(x,t)$ is a subsolution to a certain equation. This part of the proof follows the techniques of \cite{Ishii}. The second part of the proof is to use that $V$ satisfies the log-Lipschitz hypothesis (\ref{hyp:logL}) to construct a supersolution to the equation that $u(x,t)-v(y,t)$ satisfies, thus yielding the desired bound from above. The second part of the proof uses ideas from two papers that study Hamilton-Jacobi equations with coefficients that are not necessary Lipshitz. These are \cite{CIL} of Crandall, Ishii and  Lions, as well as Stromberg's \cite{S}. We provide the proof  in  Appendix \ref{second appendix}.

Next we prove existence and basic regularity for viscosity solutions of (\ref{eq:main}). 
\begin{thm}
\label{thm:existence}
Let $V$ satisfy (\ref{bdsV}), (\ref{hyp:logL}) and (\ref{hyp:L1}). Let $u_0\in L^\infty(\rr^n)$ be uniformly continuous. Then there exists a solution $u$ to (\ref{eq:main}) on $\rr^n$ with initial data $u_0$. 

Moreover, $u$ is uniformly continuous is $x$ and $t$, with modulus  that depend only on the modulus of continuity of $u_0$, $T$, and the constants $M$, $N$ in (\ref{bdsV}) and (\ref{hyp:logL}).
\end{thm}
 The main idea of the proof is that, if $V$ were ``regular enough," then simply the method of characteristics would provide a solution of (\ref{eq:main}). It turns out that the assumptions (\ref{bdsV}), (\ref{hyp:logL}) and (\ref{hyp:L1}) are enough for us to be able to regularize $V$, obtain classical solutions using the method of characteristics, and then take a limit.  In fact, the description of $u$ in terms of characteristics remains valid even after taking the limit, and we have:
 
 \begin{thm}
 \label{thm:u and ODE}
 Under the hypotheses of Theorem \ref{thm:existence}, there exists unique $X:\rr^n\times (0,T)\times (0,T)$ satisfying,
 \begin{equation}
 \label{X and V}
 X(x,s,t)=x+\int_s^t V(X(x,s,r), r)\, dr,
 \end{equation}
 and we have $u(x,t)=u_0(X(x,0,t))$.  Moreover, for every $t>0$ the map $\Phi_t$ defined by, $\Phi_t(x)=X(x,0,t)$ is Holder continuous with exponent $\exp(-Nt)$, where $N$ is the constant in (\ref{hyp:logL}). The maps $X(x,0,t)$ and $X(x,t,0)$ are inverses.
 \end{thm}
 
 This characterization of $u$ is useful to us for several reasons. First, will allow us to deduce information about the size and regularity of the set $\bdry\{n_\infty > 0\}$, which, according to Theorem \ref{main result}, is the only region on which we ``don't know" what the limit of the $p_k$ is.  Second, it connects the notion of generalized flow that we introduce in the next section and use in the rest of the paper with more classical notions. The  proofs of Theorem \ref{thm:existence} and Theorem \ref{thm:u and ODE} are in Appendix \ref{first appendix}. 

\subsection{Discontinuous viscosity solutions to (\ref{eq:main})}
It is clear that Theorem \ref{thm:comparison} implies uniqueness of continuous viscosity solutions to (\ref{eq:main}). The situation for discontinuous solutions is more subtle. In fact, as described in the introduction, there are equations for which uniqueness holds in the class of continuous solutions, but not in the class of discontinuous solutions. We use Theorem \ref{thm:comparison}, together with Theorem \ref{thm:u and ODE}, to establish existence and uniqueness of viscosity solutions to (\ref{eq:main}) with initial data a characteristic function:
\begin{thm}
\label{thm:discontinuous}
Let $\Omega_0$ be an open, bounded domain in $\rr^n$ and let $V$ satisfy (\ref{bdsV}), (\ref{hyp:logL}) and (\ref{hyp:L1}). Then there exists a viscosity solution $u$ of 
\[
u_t+Du\cdot V( x,t)=0 
\]
with initial data $u_0(x):=\chi_{\Omega_0}(x)$. Moreover, we have,
\begin{equation}\label{formula}
u(x,t)= \chi_{\Omega_t}(x), \hbox{ where }\Omega_t:=\{ x: X(x,t,0)\in \Omega_0\},
\end{equation}
where $X$ is the unique map satisfying,
 \[
 X(x,s,t)=x+\int_s^t V(X(x,s,r), r)\, dr.
 \]
And, $u$ is unique, in the sense that any other viscosity solution of (\ref{eq:main}) is between $u$ and $u^* = \chi_{\bar{\Omega_t}}$.
\end{thm}
In order to establish the uniqueness result stated here, we are taking advantage of the special form of $u$ in terms of the flow $X$. This is along the lines of the relationship pointed out in, for example, \cite[Section 2]{BSS}, between uniqueness for discontinuous solutions and the question of whether a related flow enjoys the so-called ``empty interior" property. The proof is at the end of Appendix \ref{first appendix}.

\section{Generalized flows}
\label{sec:gen flows}
We introduce a notion of generalized flows with velocity $ V(x,t)$, 
where $V$ satisfies (\ref{bdsV}), (\ref{hyp:L1}) and (\ref{hyp:logL}). Throughout, we use $\Omega^{int}$ to denote the interior of the set $\Omega$.
\begin{defn}
Let $(\Omega_t)_{t\in (a,b)}$  be a family of open subsets of $\rr^n$, and let  $V$ satisfy (\ref{bdsV}), (\ref{hyp:logL}), and (\ref{hyp:L1}). 
\begin{itemize}
\item The family $(\Omega_t)_{t\in (a,b)}$  is called a \emph{generalized superflow} with  velocity $V$ if for all $\bar{x}\in \rr^n$, $t\in (a,b)$, $r>0$, $\alpha>0$ and for all smooth functions $\phi:\rr^n\rightarrow \rr$ such that 
\[
\{x: \phi(x)\geq 0\}\subset \Omega_t\cap B_r(\bar{x}),
\]
with $|D\phi|\neq 0$ on $\{x: \phi(x)=0\}$, there exists $\bar{h}>0$ such that for all $h\in (0,\bar{h})$,
\[
\left\{x: \phi(x)-\int_t^{t+h} V( x,s) \cdot D\phi(x)\, ds -h\alpha> 0\right\}\cap \bar{B}_r(\bar{x})\subset \Omega_{t+h},
\]
and $\bar{h}$ depends only on $\alpha$, $||\phi||_{C^3(B_r(\bar{x}))}$, and the constant $M$ that appears in the hypotheses (\ref{bdsV}), (\ref{hyp:logL}), and (\ref{hyp:L1}).
\item The family $(\Omega_t)_{t\in (a,b)}$ is called a \emph{generalized subflow} with  velocity $V$ if for all $\bar{x}\in \rr^n$, $t\in (a,b)$, $r>0$, $\alpha>0$ and for all smooth functions $\phi:\rr^n\rightarrow \rr$ such that 
\[
\{x: \phi(x)\leq 0\}\subset \bar{\Omega}_t^c\cap B_r(\bar{x}),
\]
with $|D\phi|\neq 0$ on $\{x: \phi(x)=0\}$, there exists $\bar{h}>0$ such that for all $h\in (0,\bar{h})$,
\[
\left\{x: \phi(x)-\int_t^{t+h} V( x,s)\cdot D\phi(x)\, ds -h\alpha< 0\right\}\cap \bar{B}_r(\bar{x})\subset \bar{\Omega}_{t+h}^c,
\]
and $\bar{h}$ depends only on $\alpha$, $||\phi||_{C^3(B_r(\bar{x}))}$, and the constant $M$ that appears in the hypotheses (\ref{bdsV}), (\ref{hyp:logL}), and (\ref{hyp:L1}).
\end{itemize}
\end{defn}

Whether $\Omega_t$ is a generalized flow with velocity $V$ is closely related to whether $\chi_{\Omega}-\chi_{\bar{\Omega}^c}$ is a solution of (\ref{eq:main}). Precisely:
\begin{thm}
\label{thm:soln implies flow}
Let $V$ satisfy (\ref{bdsV}), (\ref{hyp:logL}), and (\ref{hyp:L1}). 
\begin{enumerate}
\item  \label{item:superflow supersoln} $(\chi_{\Omega}(x,t)-\chi_{\bar{\Omega}^c}(x,t))_*$ is a supersolution of (\ref{eq:main}) on $\rr^n\times [0,T]$ if and only if $(\Omega_t^{int})_{t\in [0,T]}$ is a generalized superflow with  velocity $V$.
\item \label{thm:subflow subsoln}
$(\Omega_t^{int})_{t\in [0,T]}$ is a generalized subflow with  velocity $V$ if and only if $(\chi_{\Omega}(x,t)-\chi_{\bar{\Omega}^c}(x,t))^*$ is a subsolution of (\ref{eq:main}) on $\rr^n\times [0,T]$.
\end{enumerate}
\end{thm}
The proof follows along the lines of the proof of Theorem 2.4 of \cite{BS}. The key idea  is that, for any (smooth enough) function $\phi(x)$, the function defined by,
\[
\psi(x,r)=\phi(x)-\int_t^r H( x, s, D\phi(x))\, ds -(r-t)\alpha/2
\] 
``should be"  a subsolution of the equation (\ref{eq:main}). Indeed, if $H$ were differentiable in $x$ and continuous in $t$, we'd have,
\[
D\psi(x,r) = D\phi(x) - \int_t^r H_p( x, s, D\phi(x))D^2\phi(x) +H_x( x, s,D\phi(x))\, ds,
\]
and hence $D\psi(x,r) = D\phi(x) + O(|r-t|)$. Taking the derivative of $\psi$ in $r$ yields,
\[
\psi_r(x,t)=H( x, r,D\phi(x))-\alpha/2,
\]
and since $D\psi(x,r) = D\phi(x) + O(|r-t|)$, we find that the right-hand side of the previous line is bounded from above by $H( x, r,D\psi(x,r))$ for $r$ close enough to $t$.

To make these ideas precise, we regularize $H$ in the space variable before carrying out this argument. We do so in the following lemma, which will also be useful to us in Section \ref{sec:superflow}. The proofs of the lemma and of Theorem \ref{thm:soln implies flow} are in Appendix \ref{third appendix}.

\begin{lem}
\label{lem:psiepsubsoln}
Let $\phi\in C^3(\rr^n)$ and let $\alpha>0$. We take $H(x,t,p)=p\cdot V(x,t)$, where $V$ satisfies (\ref{bdsV}), (\ref{hyp:logL}), and (\ref{hyp:L1}). Let  
$\rho$ be a standard bump function, supported on $B_1(0)$ and with $0\leq \rho\leq 1$ everywhere, and let $\rho_\ep(x)=\frac{1}{\ep^n}\rho\left(\frac{x}{\ep}\right)$. Define $H^\ep$ as the convolution in $x$ of $H$ and $\rho$; namely,
\[
H^\ep(x,t,p)=\int_{y\in \rr^n}H(x-y,t,p)\rho_\ep(y)\, dy.
\]
Define $\psi_\ep$ and $\bar{\psi}_\ep$ by,
\begin{equation}
\label{defpsi}
\psi_\ep(x,r)=\phi(x)-\int_t^r H^\ep( x, s, D\phi(x))\, ds -(r-t)\alpha/2.
\end{equation}
and
\[
\bar{\psi}_\ep(x,r)=\phi(x)-\int_t^r H^\ep( x, s, D\phi(x))\, ds +(r-t)\alpha/2,
\]
There exist constants $\ep_1>0$ and $\bar{h}>0$, both depending only on $\alpha$, $||\phi||_{C^3(\rr^n)}$ and the constant $N$ in (\ref{hyp:logL}), such that $\psi_{\ep_1}$ is a  subsolution and $\bar{\psi}_\ep$ is a supersolution 
of (\ref{eq:main}) on $\rr^n\times (t,t+\bar{h})$. Moreover, for this same $\ep_1$ we have, for all $x\in \rr^n$ and for all $t$,
\begin{equation}
\label{H^ep minusH}
|H^{\ep_1}(x,t,D\phi(x))-H(x,t,D\phi(x))|\leq \frac{\alpha}{4}.
\end{equation}
\end{lem}

\section{Preliminaries}
\label{sec:prelim}

Now that we have introduced the notions of viscosity solution and generalized flow that we will be using, we are almost ready to study the limit of our system. However, we first need to take care of a few preliminaries.  We  recall that $p_k$ satisfies (\ref{transport_k_pressure}). 
For the remainder of the paper  we  take $\nu=1$. Indeed, if $p_k$, $W_k$, $G$ are as in the introduction, then $\tilde{p}_k$, $\tilde{W}_k$, $\tilde{G}$ given by,
\[
\tilde{p}_k(x,t)=p_k(\sqrt{\nu}x,\nu t); \tilde{W}_k=W_k(\sqrt{\nu}x,\nu t); \tilde{G}(u) = \nu G(u),
\]
satisfy 
\begin{equation}
\label{eq:fromPandV}
\begin{cases}
\partial_t p_k-Dp_k\cdot DW_k = k p_k(W_k-p_k+ G(p_k)),\\
-\nu\Delta W_k+W_k=p_k .
\end{cases}
\end{equation} 
(We have also renumbered, so that $k+1$ becomes $k$.) Thus our assumption $\nu=1$ does not result in any loss of generality. We will focus on (\ref{eq:fromPandV}) for the remainder of the paper.

First, we recall for the reader the standard definition of weak limit that we use throughout the paper:
\begin{defn}
\label{def:liminf}
Let $\{u_k\}$ be a sequence of functions. We define:
\[
\liminf_{k\rightarrow\infty}{}_*u_k(x,t)=\lim_{k\rightarrow\infty}\inf\{u_j(y,s):\text{  } j\geq k \text{ and } |(x,t)-(y,s)|\leq k^{-1}\},
\]
and,
\[
\limsup_{k\rightarrow\infty}{}^*u_k(x,t)=\lim_{k\rightarrow\infty}\sup\{u_j(y,s): \text{  } j\geq k \text{ and } |(x,t)-(y,s)|\leq k^{-1}\}.
\]
\end{defn}

Next, we summarize the results on the system (\ref{eq:fromPandV}) that we will use from the paper of Perthame and Vauchelet \cite{PV}. 
\begin{lem}
\label{lem:results perthame}
Let $p_k$ and $W_k$ satisfy (\ref{eq:fromPandV}) with initial data satisfying (\ref{initial}). We have,
\begin{equation}
\label{Wkbd}
0\leq W_k\leq C
\end{equation}
and,
\begin{equation}
\label{bd pk}
0\leq p_k\leq P_M.
\end{equation}
In addition, there exists 
\[
W_\infty\in C(\rr^n\times (0,T))\cap L^{\infty}((0,T), W^{1,\infty}(\rr^n)),
\]
 such that $DW_\infty$ satisfies (\ref{bdsV}), (\ref{hyp:logL}) and (\ref{hyp:L1}) and, along a subsequence (still denoted by $W_k$),
\begin{equation}
\label{Vk converge}
DW_k\text{ converge strongly in $L^1((0,T), L_{loc}^\infty(\rr^n))$  to $DW_\infty$}
\end{equation}
and the $W_k$ converge to $W_\infty$ locally uniformly.
\end{lem}
\begin{proof}
Items (\ref{Wkbd}) - (\ref{Vk converge}), as well as the fact that $DW_\infty$ satisfies  (\ref{bdsV}) and (\ref{hyp:L1}), follow directly from the statement of \cite[Lemma 2.1]{PV}. That $DW_\infty$ satisfies (\ref{hyp:logL}) is a direct consequence of \cite[equation (1.14)]{PV}, which asserts that $W_\infty$ satisfies,
\[
- \Delta W_\infty+W_\infty=p_\infty \text{ on }\rr^n,
\]
where $p_\infty$ is, in particular, a function in $L^\infty$. Thus, classical results (see, for example, \cite{Stein}) yield the desired claim.

The uniform estimates $W_k\in L^\infty((0,T), W^{1,q}(\rr^d))$ and  $\partial_tW_k\in L^1((0,T), L^q(\rr^d))$, for $1\leq q\leq \infty$ of \cite[Lemma 2.1]{PV} imply that the $W_k$ are equicontinuous. According to (\ref{Wkbd}) the $W_k$ are also uniformly bounded. Hence, by the Arzela Ascoli Theorem,  we also have that the $W_k$ converge locally uniformly to $W_\infty$ along a further subsequence. In particular, we conclude $W_\infty$ is continuous.
\end{proof}

Finally, we establish some elementary properties of $G$ and of the reaction term in (\ref{eq:fromPandV}) that we will use throughout the remainder of the paper.
\begin{lem}
\label{lem:properties of G}
Suppose $G$ satisfies (\ref{G basic}) and that the family $W_k$ satisfies (\ref{Wkbd}). Then we have the bounds
\begin{equation}
\label{Gbdbelow}
G(u)\geq \balpha(P_M-u),
\end{equation}
\begin{equation}
\label{reaction term bd below}
u(W_k - u + G(u))\geq u(\balpha P_M -(1+\balpha)u)
\end{equation}
and 
\begin{equation}
\label{bdabovereac}
u(W_k - u + G(u))\leq u (1+\balpha)P_M
\end{equation}
for all $u\in [0,P_M]$. In addition, if we denote $H(u)=(u-G(u))^{-1}$, then $H'(u) \in [0,1)$ for all $u$.
\end{lem}
\begin{proof}
First we examine the function $G$. 
The properties of $G$ in (\ref{G basic}) imply, for $u\in [0,P_M]$,
\[
-G(u)=G(P_M)-G(u)=\int_u^{P_M} G'(r)\, dr \leq \int_u^{P_M} -\balpha \, dr= -\balpha(P_M-u),
\]
which implies that (\ref{Gbdbelow}) holds.
 
Next we  use (\ref{Gbdbelow}) to obtain a lower bound on our reaction term:
\[
u(W_k - u +  G(u))\geq u(W_k - u+ \balpha(P_M-u))= u(W_k+ \balpha P_M -(1+\balpha)u).
\]
Using $W_k\geq 0$, we find,
\[
u(W_k - u +  G(u))\geq u( \balpha P_M -(1+\balpha)u).
\]

To obtain the bound from above we use the estimate $W_k\leq P_M$ and obtain,
\[
u(W_k - u +  G(u))\leq u(P_M - u +  G(u)).
\]
Since $-G$ is increasing and $u\geq 0$, we have $u-  G(u)\geq -  G(0)$. Since, by (\ref{Gbdbelow}), $G(0) \geq \balpha P_M$, we have,
\begin{equation}
u(W_k - u +  G(u))\leq u (1+ \balpha)P_M.
\end{equation}

The assertion about $H$ is line (7) of \cite{PV}.
\end{proof}

\section{Superflow}
\label{sec:superflow}
Now we will establish item (\ref{item:superflow}) of Proposition \ref{prop:flows}.

In this section we  $V_k$ and $V$ to denote,
\begin{equation}
\label{notation Vk V}
V_k=DW_k \text{ and } V=DW_\infty.
\end{equation}
And, we use $f(u)$ and $\bar{a}$ to denote,
\[
f(u)=u( \balpha P_M -(1+\balpha)u),\   \  \bar{a}= \balpha P_M(1+\balpha)^{-1}.
\]
According to  (\ref{reaction term bd below}) of Lemma \ref{lem:properties of G},  $p_k$ is a supersolution of
\begin{equation}
\label{S}
u_t-Du\cdot V_k -kf(u)=0.
\end{equation}

In the following lemma we construct a barrier that we will use in this proof.
\begin{lem}
\label{lem:propagation}
Assume the hypotheses of Theorem \ref{main result}. Let $x_0\in \rr^n$, $r>0$, $t_0>0$ and let $\phi$ be a smooth function with $\{\phi\geq 0\}\subset B_r(x_0)$.  Let $a\in(0,\bar{a})$, $\beta>0$. There exists $\bar{h}>0$, that does not depend on $\beta$ or $a$, and a subsolution $Q^{k, \beta,a}(x,t)$ of (\ref{S}) in $\rr^n\times (t_0,t_0+\bar{h})$ such that 
\begin{equation}
\label{Qkbeta initial}
Q^{k, \beta,a}(\cdot, t_0)\leq a\chi_{\{\phi\geq \beta\}}
\end{equation}
for all $k$ large. And, if $(x,h)\in B_r(x_0)\times (0,\bar{h})$ is such that 
\begin{equation}
\label{conditionx}
\phi(x)+\int_{t_0}^{t_0+h}V(x, s)\cdot D\phi(x)\, ds -h\alpha >2\beta
\end{equation}
holds, then
\[
\liminf_{k\rightarrow \infty}{}_* Q^{k, \beta,a}(x,t_0+h)=a.
\]
\end{lem}

\begin{proof}[Proof of Lemma \ref{lem:propagation}]
Let us assume, without loss of generality, $t_0=0$. 

Let $(V_k)^\ep$ be the regularization of $V_k$ in space defined in Lemma \ref{lem:psiepsubsoln}. We define  $\psi(x,t)$ and $\psi^k$  by:
\[
\psi(x,t)=\phi(x)+\int_0^{t}D\phi(x)\cdot V( x,s)\, ds -t\alpha -2\beta,
\]
\[
\psi^k(x,t)=\phi(x)+\int_0^{t}D\phi(x)\cdot (V_k)^{\ep}( x,s)\, ds -t\frac{\alpha}{2}-2\beta.
\]
Notice that $\psi^k=\psi^\ep-2\beta$, where $\psi^\ep$ is  as defined in Lemma \ref{lem:psiepsubsoln}. In particular, since, according to Lemma \ref{lem:results perthame}, we have that the $V_k$ satisfy (\ref{bdsV}) and (\ref{hyp:logL}) uniformly in $k$, Lemma \ref{lem:psiepsubsoln} implies that there exists $\ep>0$ and $\bar{h}$ such that for all $k$,  $\psi^\ep$, and therefore $\psi^k$, 
is a viscosity solution of 
\begin{equation}
\label{eqpsik}
\partial_t \psi^k -D\psi^k V_k\leq 0
\end{equation}
on $\rr^n\times (0,\bar{h})$. Moreover, since each of the $V_k$ is continuous in $t$, Lemma \ref{lem:usualsoln} implies that $\psi^k$ satisfies (\ref{eqpsik}) in the classical viscosity sense.

We also recall that, according to Lemma \ref{lem:psiepsubsoln},
\[
|D\phi(x)\cdot V_k( x,s)- D\phi(x)\cdot (V_k)^{\ep}(x,s)|\leq \frac{\alpha}{4}
\]
for all $x$ and for all $s>0$. We use this to estimate the difference in size between $\psi$ and $\psi^k$:
\begin{align*}
\psi(x,h)-\psi^k(x,h)&=\int_0^h (V(x,s)-(V_k)^{\ep}(x,s))\cdot  D\phi(x)\, ds - \frac{\alpha}{2}h\\
&\leq \int_0^h (V(x,s)-V_k(x,s))\cdot  D\phi(x)\, ds - \frac{\alpha}{4}h\\
&\leq ||D\phi||_{\infty} \int_0^h |V(x,s)-V_k(x,s)| \, ds - \frac{\alpha}{4}h.
\end{align*}
Since $V_k\rightarrow V$ in $L^1((0,T), L^\infty(\rr^n)))$, we have that the right-hand side of the previous line is non-positive for all $k$ large enough and for all $x,s$. 
Thus we find,
\begin{equation}
\label{psiandpsik}
\psi(x,s)\leq \psi^k(x,s)
\end{equation}
for all $x$, for all $k$ large enough, and  for all $s>0$.

Next let  us take $q: \rr\rightarrow [0,a]$ to be a smooth non-decreasing function on $\rr$ with $q(-1)=0$ and $q(1)=a$, and define $Q^{k,\beta,a}$ by,
\[
Q^{k,\beta,a}(x,t)=q(k\psi^k(x,t)).
\]
Since $\beta$ and $a$ are fixed throughout this proof, we will no longer write them in the superscript. We remark that $f(Q^k)\geq 0$ holds since $0\leq q(\zeta) <a<\bar{a}$ for all $\zeta$ and $f\geq 0$ on $[0,\bar{a}]$.

We have,
\begin{align*}
Q^k_t-V_k(x,t)\cdot DQ^k &= \dot{q} k \psi^k_t -\dot{q}kV_k(x,t)\cdot D\psi^k\\
&= \dot{q} k [ \psi^k_t -V_k(x,t)\cdot D\psi^k]\\
&\leq 0,
\end{align*}
where the inequality holds for $t\in (0,\bar{h})$, and follows from (\ref{eqpsik}) and because $\dot{q}\geq 0$. Finally, since $f(Q^k)\geq 0$, we find,
\[
Q^k_t-V_k(x,t)\cdot DQ^k \leq kf(Q^k),
\] 
so that $Q^k$ is a subsolution to (\ref{S}) on $(0,\bar{h})$. 

Let us now check the behavior of the $Q^k$ at time 0. The definitions of $Q^k$ and $\psi^k$  yield:
\begin{equation}
\label{Qqtime0}
Q^k(x,0)=q(k (\phi(x)-2\beta)).
\end{equation}
Let us suppose $x$ is such that $\phi(x)<\beta$. This implies that, for all  $k\geq 1/\beta$,
\[
\phi(x)\leq 2\beta - k^{-1}.
\]
 The previous line holds if and only if
\[
k(\phi(x)-2\beta)\leq -1.
\]
Applying $q$, which is non-decreasing, yields,
\[
q(k(\phi(x)-2\beta))\leq q(-1)=0.
\]
Together with (\ref{Qqtime0}), this implies that if $x$ is such that $\phi(x)<\beta$, then $Q^k(x,0)=0$. In addition, $Q^k(x,0)\leq a$ for all $x$ and $t\geq 0$. Thus, we have shown that (\ref{Qkbeta initial}) holds.

Now we study the $\liminf$ of the $Q^k$. Since $q$ is non-decreasing for each $k$, (\ref{psiandpsik})  implies,
\begin{equation}
\label{qpsiqpsik}
q(k\psi(x,t))\leq q(k\psi^k(x,t))=Q^k(x,t),
\end{equation}
where the equality is simply the definition of $Q^k$. 

Now suppose  $(x,t)$ is such that (\ref{conditionx}) holds (with $t_0=0$ and $t$ instead of $h$). This says exactly that $\psi(x,t)>0$, so that $\psi(y,s)>0$ for all $(y,s)\in B_{\tilde{r}}(x,t)$ for some $\tilde{r}\geq 0$. Therefore there exists $K$ such that $k\psi(y,s)>1$  for all $(y,s)\in B_{\tilde{r}}(x,t)$ and for $k\geq K$. Applying $q$, which is non-decreasing, and then using that $q(1)=a$  yields,
\[
q(k\psi(y,s))\geq q(1)=a \text{ for all $(y,s)\in B_{\tilde{r}}(x,t)$.}
\] 
We now use (\ref{qpsiqpsik}) to estimate the left-hand side of the previous line from above and find,
\[
Q^k(y,s)\geq a \text{ for all $(y,s)\in B_{\tilde{r}}(x,t)$.}
\]
 Therefore, taking $\liminf{}_*$ gives,
\[
\liminf_{k\rightarrow\infty}{}_*Q^k(x,t) \geq    a.
\]
\end{proof}

We are now ready for:
\begin{proof}[Proof of item (\ref{item:superflow}) of Proposition \ref{prop:flows}]
Let $\bar{x}\in \rr^n$, $t\in (0,T)$, $r>0$, $\alpha>0$ and let $\phi:\rr^n\rightarrow \rr$ be a smooth function such that 
\[
\{x: \phi(x)\geq 0\}\subset (\Omega^1_t)^{int}\cap B_r(\bar{x}),
\]
with $|D\phi|\neq 0$ on $\{x: \phi(x)=0\}$. Let us use $A_{\beta,h}$ to denote,
\[
A_{\beta,h}=\left\{x| \phi(x)+\int_t^{t+h} V( x,s)\cdot D\phi(x)\, ds -h\alpha> 2\beta\right\}.
\] 
Let $\bar{h}$ be as given by Lemma \ref{lem:propagation}. We fix some $\beta>0$ for the remainder of the proof. We will establish that for $h\in(0,\bar{h})$,
\[
A_{\beta,h}\subset (\Omega^1_{t+h})^{int}.
\]
Since $\beta$ is arbitrary, establishing the previous line will complete the proof.

That $\{x: \phi(x)\geq 0\}$ is contained in $ (\Omega^1_t)^{int}\cap B_r(\bar{x})$ implies that there exists $a>0$ with $\liminf{}_*p_k\geq 2a$ on $\{x: \phi(x)\geq 0\}$, and so, for all $k$ large enough and $x$ such that $\phi(x)\geq 0$, we have $p_k(x,t)\geq a$. Since $p_k\geq 0$ everywhere, we find, for all $k$ large enough, and for all $x$, 
\begin{equation}
\label{p at time t}
p_k(x,t)\geq a\chi_{\{\phi\geq 0\}}(x,t)\geq a\chi_{\{\phi\geq \beta \}}(x,t).
\end{equation}

Let us use $Q^k(x,t)$ to denote $Q^{k, \beta,a}$  as given in Lemma \ref{lem:propagation}. 
According to (\ref{p at time t}) and (\ref{Qkbeta initial}), we have,
\[
p_k(x,t)\geq Q^k(x,t) \text{ for all  }x\in \rr^n.
\] 
In addition, we have that $Q^k$ is a subsolution of (\ref{S}) on $(t,t+\bar{h})$, and $p_k$ is a supersolution of (\ref{S}). Therefore, we have, for all $h\in (0,\bar{h})$,
\[
p_k(x,t+h)\geq Q^k(x,t+h) \text{ for all }x\in\rr^n.
\]
Now let $y\in A_{\beta,h}$. Since $A_{\beta,h}$ is an open set, there exists $\tilde{r}$ such that $B_{\tilde{r}}(y)\subset A_{\beta,h}$. Thus let take $x\in B_{\tilde{r}}(y)$ and take $\liminf_{*}$ of the previous line and find,
\[
\liminf{}_{*}p_k(x,t+h)\geq \liminf{}_{*}Q^k(x,t+h).
\]
Since $x\in A_{\beta,h}$,  Lemma \ref{lem:propagation} implies that the right-hand side of the previous line equals $a$. Therefore, $\liminf_{*}p_k(x,t+h)\geq a>0$, which means $x\in \Omega^1_{t+h}$, and hence  $y\in(\Omega^1_{t+h})^{int}$,  as desired.
\end{proof}

\section{Subflow}
\label{sec:subflow}
In this section we prove item  \ref{item:subflow} of Proposition \ref{prop:flows}. As in the previous section, we will employ the notation (\ref{notation Vk V}). In addition, we will use that, according to (\ref{bdabovereac}), $p_k$ is a subsolution of,
\begin{equation}
\label{S2}
u_t-Du\cdot V_k -ku(1+\balpha)P_M=0.
\end{equation}

We first construct suitable supersolutions to (\ref{S2}) in: 
\begin{lem}
\label{lem:propagation2}
Assume the hypotheses of Theorem \ref{main result}. Let $x_0\in \rr^n$, $t_0\geq 0$, $r>0$, and let $\phi$ be a smooth function with $\{\phi\leq 0\}\subset B_r(x_0)$.   There exists $\bar{h}>0$  and a supersolution $Q_{k}(x,t)$ of (\ref{S2}) in $\rr^n\times (t_0,t_0+\bar{h})$ such that 
\begin{equation}
\label{Qkbeta initial2}
Q_{k}(\cdot, t_0)\geq P_M\chi_{\{\phi>0\}}
\end{equation}
for all $k$. And, if $(x,h)\in B_r(x_0)\times (0,\bar{h})$ is such that 
\begin{equation}
\label{conditionx2}
\phi(y)+\int_{t_0}^{t_0+h'}V(y, s)\cdot D\phi(y)\, ds +h'\alpha <0 \text{ holds for } (y,h')\in \bar{B}_{\tilde{r}}(x,h)
\end{equation}
 for some $\tilde{r}>0$, then there exists $K$ such that
\[
Q_{k}(y,t_0+h')=0 \text{ for $(y,h')\in\bar{B}_{\tilde{r}}(x,h)$ for all }k\geq K.
\]
\end{lem}
\begin{proof}
The proof of this lemma is similar to that of Lemma \ref{lem:propagation}, so we omit some details. Without loss of generality we take $t_0=0$, and we write $t$ instead of $h$. We take $\bar{\psi}^k$ as in Lemma \ref{lem:psiepsubsoln} and $\bar{\psi}$ to be:
\[
\bar{\psi}(x,t)=\phi(x)+\int_0^{t}D\phi(x)\cdot V( x,s)\, ds +t\alpha,
\]
\[
\bar{\psi}^k(x,t)=\phi(x)+\int_0^{t}D\phi(x)\cdot (V_k)^{\ep}( x,s)\, ds +t\frac{\alpha}{2}.
\]
As in Lemma \ref{lem:propagation}, we find that  there exists $\ep>0$ and $\bar{h}$ such that for all $k$,  $\bar{\psi}^k$
is a viscosity solution (in the classical sense) of 
\begin{equation}
\label{eqpsik2}
\partial_t \bar{\psi}^k -D\bar{\psi}^k V_k\geq 0
\end{equation}
on $\rr^n\times (0,\bar{h})$. And, again similarly to Lemma \ref{lem:propagation}, we find
\begin{equation}
\label{psiandpsik2}
\bar{\psi}(x,s)\geq \bar{\psi}^k(x,s)
\end{equation}
for all $x$ and for all $s>0$.

Next let  us take $q: \rr\rightarrow [0,a]$ to be a smooth non-decreasing function on $\rr$ with $q(-1)=0$ and $q(0)=P_M$, and define $Q_{k}$ by,
\[
Q_{k}(x,t)=q(k\bar{\psi}^k(x,t))e^{k(1+\bar{\alpha}P_M)t}.
\]

We have,
\begin{align*}
\partial_tQ_k-V_k(x,t)\cdot DQ_k &= k(1+\bar{\alpha}P_M) Q_{k} +\dot{q} k \bar{\psi}^k_t -\dot{q}kV_k(x,t)\cdot D\bar{\psi}\\
&= k(1+\bar{\alpha}P_M) Q_{k}+\dot{q} k [ \bar{\psi}^k_t -V_k(x,t)\cdot D\bar{\psi}]\\
&\geq k(1+\bar{\alpha}P_M) Q_{k},
\end{align*}
where the last line holds for $t\in (0,\bar{h})$, and follows from (\ref{eqpsik2}) and because $\dot{q}\geq 0$. Thus we find that $Q_k$ is a supersolution to (\ref{S2}) on $(0,\bar{h})$.

Let us now check the behavior of the $Q^k$ at time 0. The definitions of $Q_k$ and $\bar{\phi}^k$  yield:
\begin{equation}
\label{Qqtime02}
Q^k(x,0)=q(k \bar{\phi}(x)).
\end{equation}
Let us suppose $x$ is such that $\bar{\phi}(x)>0$. Applying $q$, which is non-decreasing, yields,
\[
q(k\bar{\phi}(x))\geq q(0)=P_M.
\]
Together with (\ref{Qqtime02}), this implies that if $x$ is such that $\bar{\phi}(x)>0$, then $Q^k(x,0)=P_M$. In addition, $Q^k(x,0)\geq 0$ for all $x$ and $t\geq 0$. Thus, we have shown that (\ref{Qkbeta initial2}) holds.

Now suppose  $(x,t)$ is such that (\ref{conditionx2}) holds  on $\bar{B}_{\tilde{r}}(x,t)$ for some $\tilde{r}$. This says exactly that $\bar{\phi}(y,s)<0$ on $\bar{B}_{\tilde{r}}(x,t)$. Since $\bar{\phi}$ is continuous, there exists $\alpha>0$ so that 
\[
\bar{\phi}(y,s)<-\alpha \text{ for }(y,s)\in\bar{B}_{\tilde{r}}(x,t), 
\]
and hence there exists $K$ so that for $k\geq K$,
\[
k\bar{\phi}(y,s)\leq -1 \text{ for }(y,s)\in\bar{B}_{\tilde{r}}(x,t).
\]
Applying $q$ yields,
\[
q(k\bar{\phi}(y,s))\leq q(-1)=0,
\]
so that, upon multiplying by $e^{k(1+\bar{\alpha}P_M)s}$ we find,
\[
Q_k(y,s)=q(k\bar{\phi}(y,s))e^{k(1+\bar{\alpha}P_M)s}\leq 0,
\]
and therefore $Q_k(y,s)=0$ for $k\geq K$ and $(y,s)\in \bar{B}_{\tilde{r}}(x,t)$.
\end{proof}

We are now ready to present:
\begin{proof}[Proof of item (\ref{item:subflow}) of Proposition \ref{prop:flows}]
The claim of this proposition is that $((\bar{\Omega}^2_t)^c)_t$ is a generalized subflow. Let us recall that $\bar{\Omega}_t^c$ appears in the definition of $\Omega$ being a generalized subflow. Taking $\Omega_t = ((\bar{\Omega}^2_t)^c)_t$ yields
\[
\bar{\Omega}_t^c = (\Omega^2_t)^{int}.
\]
Thus, let us take $\bar{x}\in \rr^n$, $t_0\in (0,T)$, $r>0$, $\alpha>0$, and a smooth function $\phi$ such that 
\begin{equation*}
\{x: \phi(x)\leq 0\}\subset (\Omega^2_t)^{int}\cap B_r(\bar{x}),
\end{equation*} 
and   with $|D\phi|\neq 0$  on $\{\phi=0\}$. Let us use $E_{h}$ to denote,
\[
E_{h}=\left\{x| \phi(x)+\int_t^{t+h} V( x,s)\cdot D\phi(x)\, ds +h\alpha< 0\right\}.
\] 
Let $\bar{h}$ be as given by Lemma \ref{lem:propagation2}.  We will establish $E_{h}\subset (\Omega^2_{t+h})^{int}$ for  for $h\in(0,\bar{h})$.

That $\{x: \phi(x)\leq 0\}$ is contained in $ (\Omega^2_t)^{int}\cap B_r(\bar{x})$ implies that there exists $K$ such that  $p_k(x,t)=0$ on $\{x: \phi(x)\leq 0\}$ for all $k\geq K$. Since $p_k\leq P_M$ everywhere (this is exactly equation (\ref{bd pk}) of Lemma \ref{lem:results perthame}), we find, for all $k\geq K$, and for all $x$, 
\begin{equation}
\label{p at time t2}
p_k(x,t)\leq P_M\chi_{\{\phi > 0\}}(x,t).
\end{equation}

Let $Q_{k}$ be as given in Lemma \ref{lem:propagation2}. 
According to (\ref{p at time t2}) and (\ref{Qkbeta initial2}), we have,
\[
p_k(x,t)\leq Q_k(x,t) \text{ for all  }x\in \rr^n.
\] 
In addition, we have that $Q^k$ is a supersolution of (\ref{S2}) on $(t,t+\bar{h})$, and $p_k$ is a subsolution of (\ref{S2}). Therefore, we have, for all $h\in (0,\bar{h})$,
\[
p_k(x,t+h)\leq Q^k(x,t+h) \text{ for all }x\in\rr^n.
\]

Now let $x\in E_{h}$. Since 
\[
(x,h)\mapsto \phi(x)+\int_t^{t+h} V( x,s)\cdot D\phi(x)\, ds +h\alpha
\]
is continuous, we find there exists $\tilde{r}$ such that $y\in E_{h'}$ for $(y,h')\in \bar{B}_{\tilde{r}}(x,h)$.  Thus, according to Lemma \ref{lem:propagation2}, there exists $K$ such that for all $k\geq K$ and all $(y,h')\in \bar{B}_{\tilde{r}}(x,h)$, we have $Q_k(y,t+h')=0$. The previous line therefore implies 
\[
p_k(y,t+h') = 0 \text{ for }(y,h')\in \bar{B}_{\tilde{r}}(x,h), k\geq K.
\]
We recall $A_{t+h'}^k=\{x|p_k(x,t+h')>0\}$, so from the previous line we find,
\[
dist(y,A^k_{t+h'})\geq \tilde{r}/2 \text{ for  }(y,h')\in B_{\tilde{r}/2}(x,h), k\geq K,
\]
and hence 
\[
\liminf{}_*dist(y,A^k_{t+h})>0 \text{ for  }y\in B_{\tilde{r}/4}(x).
\]
This means exactly $y\in (\Omega^2_{t+h})^{int}$. Thus $E_{h}\subset(\Omega^2_{t+h})^{int}$, as desired.
\end{proof}

\section{Limiting behavior at the initial time}
\label{sec:time 0}
We established that, for $t>0$, $(\Omega_t^1)^{int}$ is a superflow, and $(\bar{\Omega_t^2})^c$ is a subflow, with  velocity $-DW_\infty$. In this section we will study these sets at the initial time $t=0$.

\begin{prop}
\label{prop:inittime}
Under the assumptions of Theorem \ref{main result}, 
\begin{enumerate}
\item $\Omega_0^{int}\subset (\Omega_0^1)^{int}$, and 
\item $(\bar{\Omega_0})^c \subset (\Omega^2)^{int}_0$.
\end{enumerate}
\end{prop}

\subsection{Positive region}
We prove the first part of Proposition \ref{prop:inittime}. We use $f$ to denote,
\[
f(u)=u( \alpha P_M -(1+\alpha)u).
\]
We will construct a barrier from the solution to the ODE described in the following lemma: 
\begin{lem}
\label{lemODE}
For each $\xi\in [0,\infty)$, there exists a unique solution $\omega(\xi,t)$ of the ODE,
\[
\omega_t=f(\omega) \text{ for }t\in [0,\infty); \, \, \omega(0)=\xi,
\]
and with $\omega_\xi(\xi, s)>0$ and $w(\zeta, s)>0$ in $(0,\infty)\times [0,\infty)$.
\end{lem}
\begin{proof}
This ODE has the solution: $\omega(0,t)\equiv 0$, and for $\xi>0$ we have,
\[
\omega(\xi,t)=\frac{  \balpha P_M}{1+\balpha+( \balpha P_M\zeta^{-1}-(1+\balpha))e^{-\balpha   P_M t}}.
\]
From this we can explicitly verify that $\omega$ has the desired property.
\end{proof}
\begin{proof}[Proof of first part of Proposition \ref{prop:inittime}]
According to (\ref{reaction term bd below}), we have that $p_k$ is a supersolution of,
\[
u_t-V_k\cdot Du = k f(u).
\]

Let $x_0\in (\Omega_0)^{int}$, so that there exists $r>0$ with $B_{2r}(x_0)\subset \Omega_0$.  Assumption (\ref{initial}) implies that there exists $a>0$ such that, for $k$ large enough and $x\in B_{2r}(x_0)$, we have $p_k(x,0)\geq a$. Let $\psi$ be a smooth function with,
 \[
0\leq \psi\leq a \text{ in }\rr^n,\,  \,  \psi=0 \text{ on } \rr^n\setminus B_{2r}(x_0),\, \, \psi=a \text{ on }B_r(x_0).
\]
We remark that $\psi$ is chosen so that,
\begin{equation}
\label{psi and uk}
\psi(x)\leq p_k(x,0) \text{ for all }x\in \rr^n \text{ and all $k$ large enough.}
\end{equation}
Let $\omega(\xi, t)$ be as in Lemma \ref{lemODE}. We define $w_k$ by,
\[
w_k(x,t)= \omega( (\psi(x)-K t)^+, kt),
\]
where we take $K$ to be $K=1/(M\sup{|D\psi|})$. 

We have,
\[
\partial_tw_k -V_k\cdot Dw_k -kf(w_k)=-K\omega_\xi+k\omega_t- V_k\cdot (\omega_\xi D\psi)-kf(\omega).
\]
The sum of the second and fourth terms is zero, due to the ODE that $\omega$ satisfies. Thus we find,
\[
\partial_tw_k +V_k\cdot Dw_k -kf(w_k)=\omega_\xi(-K-V_k\cdot D\psi).
\]
Since $\omega_\xi>0$, our choice of $K$ implies that the right-hand side of the previous line is non-positive, and thus $w_k$ is a subsolution to the equation for $u_k$. In addition, at time 0 we have,
\[
w_k(x,0)= \omega(\psi(x), 0)=\psi(x).
\]
Together with (\ref{psi and uk}), this implies that, for all $k$ large enough, $w_k(x,0)\leq p_k(x,0)$ for all $x\in \rr^n$. The comparison principle thus implies that for all $k$ large enough,
\begin{equation}
\label{wlequ}
w_k(x,t)\leq p_k(x,t) \text{ for all }(x,t)\in \rr^n\times (0, \infty).
\end{equation}

We will now establish an appropriate bound from below on $w_k$, and then use (\ref{wlequ}) to deduce the desired estimate on the limit of the $p_k$.

By definition of $\psi$, we have $\psi(x)=a$ for $x\in B_r(x_0)$. Together with the definition of $w_k$ this implies, for $x\in B_r(x_0)$,
\[
w_k(x,t)=\omega((a - Kt)^+, kt). 
\]
For $t\leq a K^{-1}/2$, we have $a - Kt\geq a/2$. Since $\omega$ is non-decreasing in $\xi$, we  find that for $x\in B_r(x_0)$ and for $t\leq a K^{-1}/2$,
\[
\omega((a - Kt)^+, kt)\geq \omega(a/2, kt). 
\] 
In addition, since  $\omega$ is non-decreasing in $t$, we have, for all $t$,
\[
\omega(a/2, kt)\geq \omega(a/2, 0) =a/2,
\]
where the equality follows since $\omega$ satisfies $\omega(\xi, 0)=\xi$. Putting the three previous lines together thus yields,  for $x\in B_r(x_0)$ and for $t\leq a K^{-1}/2$,
\[
w_k(x,t)\geq a/2.
\]
We use this to bound the left-hand side of (\ref{wlequ}) from below and find,
\[
a/2\leq p_k(x,t) \text{ for $ x\in B_r(x_0)$ and for $t\leq a K^{-1}/2$}.
\] 
Thus we have for $x\in B_{r/2}(x_0)$,
\[
\liminf{}_*p_k(x_0,0)\geq a/2>0,
\]
which means $x_0\in (\Omega_0^1)^{int}$, as desired.
\end{proof}

\subsection{Zero region}

In this subsection we study the behavior of $\limsup{}^*p_k$ at time zero and establish the second part of Proposition \ref{prop:inittime}. We devote the following lemma to the construction of a necessary barrier. 
\begin{lem}
\label{lem:barrier initial time}
Let   $r>0$ and $x_0\in \rr^n$. The equation
\[
v_t-M|Dv|=0
\]
has a solution $\bar{v}(x,t)$ on $B_{2r}(x_0)\times [0,\frac{r}{2M}]$ that satisfies,
\[
\bar{v}(x,0)\equiv 0 \text{ on }B_{r}(x_0),
\]
\[
\bar{v}(x,t) = P_M \text{ on }\bdry B_{2r}(x_0)\times [0,\frac{r}{2M}],
\]
\[
\bar{v}(x,t)\equiv 0 \text{ on }B_{r/2}(x_0)\times (0,\frac{r}{2M}).
\]
\end{lem}
\begin{proof}[Proof of Lemma \ref{lem:barrier initial time}]
Let $\phi:\rr\rightarrow [0, P_M]$ be a smooth and non-decreasing function satisfying,
\[
\phi(\xi)=0\text{ for }\xi\leq r, \   
\phi(\xi)=P_M \text{ for }\xi\geq 2r,
\]
and define,
\[
\bar{v}(x,t)=\phi(|x-x_0|+Mt).
\]
We have,
\[
D_x\bar{v}(x,t)=\phi'(|x-x_0|+Mt)\frac{x-x_0}{|x-x_0|}
\]
and 
\[
\bar{v}_t(x,t)=M\phi'(|x-x_0|+Mt),
\]
so that
\[
\bar{v}_t(x,t)-M|D\bar{v}(x,t)|= \phi'(|x-x_0|+Mt)(M - M)=0.
\]
We also remark that, although $|x-x_0|+Mt$ is not differentiable at $x=x_0$, we still have that $\bar{v}(x,t)$ is differentiable at $x_0$ so long as $Mt \leq r$, which is the case in the region we consider.

Let us now verify that $\bar{v}$ satisfies the desired properties. 

If $x\in B_r(x_0)$, then $\bar{v}(x,0)=\phi(|x-x_0|)\leq \phi(r)=0$. Therefore $\bar{v}(x,0)\equiv 0$ on $B_r(x_0)$. 

Now let us take $x\in \bdry B_{2r}(x_0)$. We have,
\[
\bar{v}(x,t)=\phi(2r+Mt)\geq \phi(2r)=P_M,
\]
where the inequality follows because $\phi$ is non-decreasing and holds for any $t>0$. Thus, $\bar{v}(x,t) \equiv P_M$ on $\bdry B_{2r}(x_0)\times [0,\frac{r}{2M}]$.

For the third property, we note that for $x\in B_{r/2}(x_0)$ and $ t\in (0, r/2M)$ we have
\[
|x-x_0|+Mt\leq r/2+r/2\leq r,
\]
so $\bar{v}(x,t)=\phi(|x-x_0|+Mt)\leq \phi(r)=0$.

\end{proof}

\begin{proof}[Proof of second part of Proposition \ref{prop:inittime}]
Let us take $x_0\in (\bar{\Omega_0})^c$. Let $r>0$ be such that $\bar{B}_{2r}(x_0)\subset (\bar{\Omega_0})^c$. Our assumption (\ref{initial}) on $p_k(x,0)$ implies that there exists $K$ large so that for $k\geq K$ we have $p_k(x,0)=0$ for all $x\in \bar{B}_{2r}(x_0)$.

Let $\bar{v}$ be as in Lemma \ref{lem:barrier initial time}, and define,
\[
w_k(x,t)=\bar{v}(x,t)e^{(1+ \balpha P_M)kt}.
\]
We will show that $w_k$ is a supersolution of the equation for $p_k$, and that $p_k\leq w_k$ on the parabolic boundary of $B_{2r}(x_0)\times (0,\frac{r}{2M})$ for $k\geq K$. Let us check the latter:
 
We have $p_k(x,0)=0$ for all $x\in \bar{B}_{2r}(x_0)$ and $k\geq K$. Since $\bar{v}(x,0)\geq 0$, we therefore find
\[
w_k(x,0)\geq p_k(x,0) \text{ on }B_{2r}(x_0).
\]
Now let $x\in \partial B_{2r}(x_0)$ and $t\in (0,\frac{r}{2M})$. According to Lemma \ref{lem:barrier initial time}, $\bar{v}(x,t)=P_M$. Therefore,
\[
w_k(x,t)= P_M e^{(1+ \balpha P_M)kt} \geq P_M \geq p_k(x,t),
\]
where the second inequality follows from the uniform bound on $p_k$ (\ref{bd pk}) of Lemma \ref{lem:results perthame}. 

Now we check that $w_k$ is a supersolution to the  equation for $p_k$. To this end, we compute:
\[
\partial_tw_k=(\bar{v}_t +\bar{v}(x,t)(1+ \balpha P_M)k) e^{(1+ \balpha P_M)kt} 
\]
and 
\[
-V_k\cdot Dw_k =-e^{(1+ \balpha P_M)kt} V_k\cdot D\bar{v}.
\]
We use the uniform bound (\ref{bdsV}) on $V_k$  of Lemma \ref{lem:results perthame} to estimate the right-hand side of the previous side from below by $-M|D\bar{v}|$ and find,
\[
-V_k\cdot Dw_k \geq -e^{(1+ \balpha P_M)kt} M|D\bar{v}|.
\]
Putting this together with $\partial_t w_k$ yields
\begin{align*}
\partial_tw_k - V_k\cdot Dw_k &\geq (1+ \balpha P_M)k e^{(1+ \balpha P_M)kt}+(\bar{v}_t -M |D\bar{v}|)e^{(1+ \balpha P_M)kt} \\
&\geq (1+ \balpha P_M)k w_k,
\end{align*}
where the last inequality follows because of the equation that $\bar{v}$ satisfies. By the estimate (\ref{bdabovereac}) on the reaction term of the equation for $p_k$, we have that $p_k$ is a subsolution of,
\[
u_t-V_k\cdot Du=(1+ \balpha P_M)ku.
\]
 We have seen that  $p_k(x,0)\leq w_k(x,0)$ holds on the parabolic boundary of $B_{2r}(x_0)\times (0,\frac{r}{2M})$. The comparison principle therefore implies,
\[
p_k(x,t)\leq w_k(x,t) \text{ on $B_{2r}(x_0)\times (0,\frac{r}{2M})$ and for $k\geq K$.}
\]
Let us now take $x\in B_{r/2}(x_0)$ and $t\in (0, r/2M)$. According to Lemma \ref{lem:barrier initial time}, we have $\bar{v}(x, t)=0$. Therefore, the definition of $w_k$ says that we have $w_k(x,t)=0\cdot e^{(1+ \balpha P_M)kt}=0$. The previous line therefore implies
\[
p_k(x,t)\leq 0 \text{ on $B_{r/2}(x_0)\times (0,\frac{r}{2M})$ and for $k\geq K$.}
\]
Therefore, 
\[
dist(x,A^k_t)\geq r/4M \text{ for }(x,t)\in B_{r/4}(x_0)\times ( 0,\frac{r}{4M}) \text{ and for }k\geq K.
\]
This implies,
\[
\liminf{}_{*}dist(x,A^k_0)>0 \text{ on } B_{r/4}(x_0),
\]
so in particular $x_0\in (\Omega^2)^{int}_0$, as desired.
\end{proof}

\section{Convergence in the positive region}
\label{sec:positive region}
So far we have shown that the region 
\[
\Omega^1_t=\{x|\liminf{}_*p_k(x,t)>0\}
\]
is a superflow with  velocity $DW_{\infty}$. Now we will establish:
\begin{prop}
\label{prop:conv in pos region}
Let $Q$ be a compact subset of $\{(x,t)|t>0, x\in (\Omega^1)^{int}_t\}$. Then $p_k$ converges uniformly on $Q$ to $(Id-  G)^{-1}(W_\infty(x,t))$.
\end{prop}

\begin{proof}

Let $(x_0,t_0)\in Q$. We will establish,
\[
\liminf{}_*p_k(x_0,t_0)\geq (Id-G)^{-1}(W_\infty(x_0,t_0))\geq \limsup{}^*p_k(x_0,t_0),
\]
which implies the desired result.

There exists $r_0$ with $B_{r_0}(x)\subset \Omega^1_t$ for $|t-t_0|\leq r_0$, and $t_0-r_0>0$. Thus there exists $a>0$ such that, for $k$ large enough, 
\begin{equation}
\label{pk on Omega1}
p_k(x,t)\geq a/2 \text{ for }(x,t)\in Q_{r_0}(x_0,t_0),
\end{equation}
 where we use $Q_r(x_0,t_0)$ to denote the cylinder, $Q_r(x_0,t_0)=B_r(x_0,t_0)\times [t_0-r,t_0+r]$. 

Let us fix $\ep>0$. Due to the uniform convergence of the $W_k$ (see Lemma \ref{lem:results perthame}), we have that there exist $r\in (0,r_0)$  such that for all $k$ large enough, 
\begin{equation}
\label{Wk Winfty}
|W_k(x,t)- W_\infty(x_0,t_0)|\leq \ep \text{ for }(x,t)\in Q_r(x_0,t_0).
\end{equation}

For the remainder of this proof we use $H(u)$ and $\beta$ to denote,  respectively, $H(u)=(Id-G)^{-1}(u)$ and $\beta=W_\infty(x_0,t_0)$. 

\textbf{Step one.} 
In this step we will establish,
\begin{equation}
\label{limsup pk}
\limsup{}^{*}p_k(x_0,t_0)\leq H(\beta).
\end{equation}
Using (\ref{Wk Winfty}), together with the uniform bound (\ref{bd pk}), yields that $p_k$ is a subsolution of,
\begin{equation}
\label{pk subsolution}
p_k-Dp_kDW_k\leq kP_M(W_\infty(x_0,t_0)+\ep -p_k+G(p_k))
\end{equation}
on $Q_r(x_0,t_0)$, where we assume without loss of generality $P_M\geq 1$.

 Let $\psi:\rr^+\rightarrow \rr$ be a smooth increasing function such that,
\[
1\leq \psi\leq P_M,\   \   \psi(z)\equiv 1 \text{ for }z\leq r/4,\   \   \psi(z)\equiv \frac{P_M}{H(\beta)} \text{ for }z\geq r/2.
\]
Define 
\[
g(t)=H(\beta)+P_M e^{-k (t-(t_0-r/8M))}+\ep,
\]
and
\[
v(x,t)= g(t) \psi(|x-x_0|+M(t-(t_0-r/8M))).
\]
First we establish that $v\geq p_k$ on the parabolic boundary of $Q:=B_{r}(x_0)\times (t_0-r_0/8M, t_0+r_0/8M)$. To this end, we first take $x\in \bdry B_{r_0}(x_0)$ and $t\in (t_0-r_0/8M, t_0+r_0/8M)$. We have,
\begin{align*}
|x-x_0|+M(t-(t_0-r/8M))&\geq |x-x_0|= r.
\end{align*}
Since $\psi$ is increasing, we find 
\[
\psi(|x-x_0|-M(t-(t_0-r/8M))\geq \psi(r)= \frac{P_M}{H(\beta)} ,
\]
where the equality follows from our choice of $\psi$. Thus we have,
\[
v(x,t)\geq g(t)\frac{P_M}{H(\beta)} \geq P_M.
\]
Next we take $x\in B_{r}(x_0)$ and look at the initial time $t_0-r/8M$, to find,
\[
v(x,t_0-r/8M)\geq g(t_0-r/8M)=(H(\beta)+P_M+\ep)\geq P_M
\]
where the first inequality follows since $\psi\geq 1$ everywhere.

Next we will show that $v$ is a supersolution to (\ref{pk subsolution}).  
We have,
\begin{align*}
v_t-DvDW_k&=g'(t)\psi+M g(t)\psi'-g(t)(\psi')(x-x_0)\cdot W_k\\
&= g'(t)\psi + g(t)(\psi')(M-(x-x_0)\cdot DW_k),\\
&=g'(t)\psi + g(t)|\psi'| (M-(x-x_0)\cdot DW_k),
\end{align*}
where $\psi$ and $\psi'$ are evaluated at $(|x-x_0|+M(t-(t_0-r/8M))$ throughout, and the first inequality follows since $\phi$ is decreasing. According to the uniform supremum bound on $DW_k$ given in (\ref{bdsV}), we have $|(x-x_0)\cdot DW_k|\leq rM\leq M$, so we find,
\[
v_t-DvDW_k\geq g'(t)\psi.
\]
Thus,
\begin{align*}
v_t-DvDW_k-&kP_M(\beta+\ep -v+G(v))\geq g'(t)\psi-kP_M(\beta+\ep -v+G(v))\\
&= -k P_M e^{-\delta k (t-(t_0-r/8M))})\psi - kP_M(\beta +\ep - (Id-G)(v)).
\end{align*}
We would like to show that the right-hand side of the previous line is non-negative (and thus $v$ is a supersolution of (\ref{eq pk super})). 
In order to do this, we need to estimate the difference between $(Id-G)(v)$ and $\beta = (Id-G)(H(\beta))$. First, since $(Id-G)$ is increasing and $\phi\geq 1$, we find,
\[
(Id-G)(v)\geq (Id-G)(g)= (Id-G)(H(\beta)+P_M e^{- k (t-(t_0-r/8M))})+\ep).
\]
Next, we recall that $G$ is $C^1$ and, moreover, $G'\leq -\alpha$. We use this to estimate the right-hand side of the previous line from below and find,
\[
(Id-G)(v)\geq (Id-G)(H(\beta))+(1+\alpha)(P_M e^{- k (t-(t_0-r/8M))})+\ep).
\]
Recalling that $(Id-G)$ and $H$ are inverses, we find
\[
(Id-G)(v)\geq \beta+(1+\alpha)(P_M e^{- k (t-(t_0-r/8M))})+\ep).
\]
Thus we have,
\begin{align*}
\beta +\ep - (Id-G)(v)&\leq (\beta +\ep)-\beta -(1+\alpha)(P_M e^{- k (t-(t_0-r/8M))})+\ep)\\
&\leq -(1+\alpha)(P_M e^{- k (t-(t_0-r/8M))}).
\end{align*}
We use this to estimate from below the last term on the right-hand side of the equation for $v$ and find,
\begin{align*}
v_t-DvDW_k-&kP_M(\beta+\ep -v+G(v))\geq \\
 &\geq -k P_M e^{-\delta k (t-(t_0-r/8M))})\psi + kP_M (1+\alpha)(P_M e^{-k (t-(t_0-r/8M))})\\
&= kP_Me^{- k (t-(t_0-r/8M))} (-\psi +(1+\alpha)P_M)) .
\end{align*}
Since $\psi\leq P_M$, we find,
\[
v_t-DvDW_k-kP_M(\beta+\ep -v+G(v)) \geq 0,
\]
and hence $v$ is a supersolution of (\ref{pk subsolution}). Since $v\geq p_k$ on the parabolic boundary of $Q$, we find $v\geq p_k$ on all of  $Q$. 

Consider $(x,t)\in Q_{r/16M}(x_0,t_0)\subset Q$. Then $t\in (t_0-r/16M, t_0+r/16M)$, and so, since $g$ is decreasing in $t$,
\[
g(t)\leq H(\beta)+P_M e^{- k (t_0-r/16M-(t_0-r/8M))}+\ep \leq  H(\beta)+P_M e^{-k (r/16M)}+\ep.
\]
In addition,
\[
|x-x_0|+M(t-(t_0-r/8M))\leq r/16M + M(r/16M+r/8M)\leq r/4
\]
where we use $M\geq 1$. Thus, since $\psi$ is increasing,
\[
\psi(|x-x_0|+M(t-(t_0-r/8M)))\leq \psi(r/4)=1.
\]
Therefore, for $(x,t)\in Q_{r/16M}(x_0,t_0)$, we have,
\[
p_k(x,t)\leq v(x,t)\leq H(\beta)+P_M e^{- k (r/16M)}+\ep.
\]
Taking $\limsup{}^*$ thus yields,
\[
\limsup{}^*p_k(x,t)\leq H(\beta)+\ep
\]
for $(x,t)\in Q_{r/16M}(x_0,t_0)$, where $r$ depends on $\ep$. Since $(x_0,t_0)\in Q_{r/16M}(x_0,t_0)$ for any $r$, we find that
\[
\limsup{}^*p_k(x_0,t_0)\leq H(\beta)+\ep
\]
holds for any $\ep>0$. Thus we conclude that the desired result (\ref{limsup pk}) holds.

\textbf{Step two.} Next we establish,
\[
\liminf{}_*p_k(x_0,t_0)\geq H(\beta).
\]
The proof is very similar to that of (\ref{limsup pk}), and we provide only a sketch. 
Using (\ref{Wk Winfty}) to bound from below the right-hand side of the equation for $p_k$ implies that $p_k$ is a supersolution of,
\[
u_t-DuDW_k=k p_k(\beta-\ep - u+G(u)).
\]
In addition, $H(\beta-\ep)$ is a solution of the equation in the previous line. Thus,
\[
\tilde{p_k}:=\min\{p_k, H(\beta-\ep)\}
\]
is also a  supersolution of that equation (here we mean supersolution in the classical viscosity sense, which suffices for the remainder of this proof). According to (\ref{pk on Omega1}), we have $p_k\geq a/2$ on $Q_r(x_0,t_0)$. This implies 
\[
\limsup{}^*p_k(x_0,t_0)\geq a/2.
\]
Thus, the inequality (\ref{limsup pk}) that we established in the first step yields,
\[
H(\beta)\geq a/2.
\]
Hence, for $\ep$ small enough, $H(\beta-\ep)\geq a/4$. Therefore, $\tilde{p}_k\geq a/4$ on $Q_r(x_0,t_0)$ as well. Therefore, $\tilde{p}_k$ is a supersolution of,
\begin{equation}
\label{eq pk super}
u_t-DuDW_k=k (W_\infty(x_0,t_0)-\ep - u+G(u))a/4
\end{equation}
on $Q_r(x_0,t_0)$. We will now construct a certain subsolution to this equation. 
 Let $\phi:\rr^+\rightarrow \rr$ be a smooth decreasing function such that,
\[
0\leq \phi\leq 1,\   \   \phi(z)\equiv 1 \text{ for }z\leq r/8,\   \   \phi(z)\equiv 0 \text{ for }z\geq r/2.
\]
Define $h(t)=H(\beta)(1-e^{-\frac{a}{4}k (t-(t_0-r/8M))})-\ep$ and 
\[
w(x,t)=h(t)\phi(|x-x_0|+M(t-(t_0-r/8M))).
\]
Just as in the previous step, we find $w\leq p_k$ on the parabolic boundary of $Q$, and that $w$ is a subsolution of (\ref{eq pk super}). Since $\tilde{p}_k$ is a supersolution of (\ref{eq pk super}) on $Q$, and $w\leq p_k$ on the parabolic boundary of $Q$, the comparison principle implies,
\[
\tilde{p}_k \geq w \text{ on }Q.
\]
Similarly to the previous step, we find that for $(x,t)\in Q_{r/16M}(x_0,t_0)$,
\[
w(x,t)\geq H(\beta)(1-e^{-\frac{a}{4}k(9r/16M))})-\ep.
\]
The two previous lines, together with the definition of $\tilde{p_k}$ imply,
\[
p_k(x,t)\geq H(\beta)(1-e^{-\frac{a}{4}k(9r/16M))})-\ep \text{ on }Q_{r/16M}(x_0,t_0).
\]
Taking $\liminf{}_*$ yields,
\[
\liminf{}_* p_k(x,t)\geq H(\beta)-\ep \text{ on }Q_{r/16M}(x_0,t_0),
\]
where $r$ depends on $\ep$. However, since $(x_0,t_0)\in Q_{r/16M}(x_0,t_0)$ for all $r$, we find 
\[
\liminf{}_* p_k(x_0,t_0)\geq H(\beta)-\ep
\]
holds for all $\ep$, and the desired result thus follows.
\end{proof}

\section{Proof of the main result}
\label{sec:proof main}
This section puts together the results of the previous ones in order to establish  our main result. Throughout this section we will use the auxiliary function $\theta$, which we define to be the unique solution of 
\begin{equation}
\label{eq:theta}
\theta_t-D\theta\cdot DW_\infty =0
\end{equation}
and with initial data,
\begin{equation}
\label{def:dist}
d(x,\bdry\Omega_0)=
\begin{cases}
dist(x, \bdry\Omega)\wedge 1 \text{ for } x\in \Omega_0,\\
-dist(x, \bdry\Omega)\vee -1 \text{ for } x\notin \Omega_0.
\end{cases}
\end{equation}
(According to Theorems \ref{thm:existence} and \ref{thm:comparison}, as well as the  assumption that $\Omega_0$ is compact, $\theta$ is well-defined.) 

For the sake of presentation, we prove items (b) and (c) of Theorem \ref{main result} first. Then we will establish some corollaries, and finally we will deduce item (a).

\begin{proof}[Proof of  parts (b), (c) of Theorem \ref{main result}]
We establish:
\begin{itemize}
\item If $Q$ is a compact subset of $\{(x,t)|t>0, \theta(x,t)>0\}$, then the $p_k$ converge uniformly on $Q$ to $H^{-1}(W_\infty)$, and the $n_k$ converge uniformly on $Q$ to 1.
\item If $Q$ is a compact subset of $\{(x,t)|t>0, \theta(x,t)<0\}$, then there exists $K$ large enough such that if $k\geq K$ then $p_k\equiv 0$ on $Q$ and $n_k\equiv 0$ on $Q$.
\end{itemize}
Together, the two bullet points imply that $p_k$ converge locally uniformly to $p_\infty=H^{-1}(W_\infty)\chi_{\{p_\infty>0\}}$, and $n_k$ to $\chi_{\{p_\infty>0\}}$, on $(\rr^n\times (0,\infty))\setminus \{\theta=0\}$. Moreover, this identifies $\{\theta>0\}$ with $\{n_\infty>0\}$ and  $\{\theta=0\}$ with $\bdry\{n_\infty>0\}$.

We establish the first bullet-point. 
Let us fix $\ep>0$. We take $\phi^{\ep}$ to be a smooth, non-decreasing function such that
\[
\phi^{\ep}(u)=
\begin{cases}
1\text{ for }u>2\ep,\\
-1 \text{ for }u<\ep,
\end{cases}
\]
and define $v_\ep$ by,
\[
v_\ep(x,t)=\phi^{\ep}(\theta(x,t)).
\]
Since $\theta$ is a viscosity solution of (\ref{eq:theta}), and $\phi^{\ep}$ is non-decreasing, a direct computation implies that  $v_\ep(x,t)$ 
is also a subsolution of (\ref{eq:theta}). (Indeed, here we are using that the equation (\ref{eq:theta}) is geometric. See, for example, \cite[Lemma 1.3]{Souganidis Front Propagation} for further discussion of this property in a more general context.)

According to item (\ref{item:superflow}) of Proposition \ref{prop:flows}, if $\Omega^1_t$ is given by, $\Omega^1_t=\{x|\liminf{}_*p_k>0\}$ 
then $(\Omega^1_t)^{int}$ is a generalized superflow with  velocity $-DW_\infty$.  Theorem \ref{thm:soln implies flow} thus implies that $w$ given by, $w(x,t)=\chi_{(\Omega_t^1)^{int}}(x) - \chi_{\bar{\Omega_t^1}^c}(x)$ is a viscosity supersolution of (\ref{eq:theta}).

We now aim to establish,
\[
v_\ep(x,0)\leq w(x,0).
\]
To this end, let us take $x$ such that $v_\ep(x,0)>-1$ (for any other $x$ we have that the previous line automatically holds, as $w\geq -1$ everywhere). The definition of $v^\ep$ implies that $x$ is such that $\theta(x,0)\geq \ep$.  The definition of $\theta(x,0) $ therefore implies $x\in  \Omega_0^{int}$. Applying Proposition \ref{prop:inittime}, we find $x\in  (\Omega_0^1)^{int}$. Therefore $w(x,0)=1\geq v_\ep(x,0)$, as desired.

We may now apply the comparison principle of Theorem \ref{thm:comparison} to find 
\begin{equation}
\label{V and W}
v_\ep(x,t)\leq w(x,t)
\end{equation}
for all $t$. 

Let $Q$ be a compact subset of $\{(x,t)|t>0, \theta(x,t)>0\}$. Since (according to Theorem \ref{thm:existence}) $\theta$ is continuous, there exists $\delta>0$ such that $\theta(x,t)\geq \delta$ for $(x,t)\in Q$.  Because  the $\phi^{\ep}$ are non-decreasing, we find, for $(x,t)\in Q$,
\[
v_\ep(x,t)=\phi^{\ep}(\theta(x,t))\geq \phi^{\ep}(\delta).
\]
Let us take $\ep=\delta/2$. Then the right-hand side of the previous line equals 1. Thus  $v_\ep(x,t)=1$ for $(x,t)\in Q$. Now we use (\ref{V and W}) to find $w=1$ on $Q$, which implies $Q\subset (\Omega_1^t)^{int}$. Hence Proposition \ref{prop:conv in pos region} implies that the $p_k$ converge uniformly to $(Id-G)^{-1}(W_\infty)$ on $Q$. 

The proof of the first bullet-point is now complete, and the statement in the second one is proved similarly. In particular, we use the definition of $\Omega^2_t$, the fact that it is a generalized subflow (proved in Proposition \ref{prop:flows}), and the second part of Proposition \ref{prop:inittime}. We omit the details.

\end{proof}

\subsection{Proof of part (a) of the main result}
\label{sec:further}

To establish part (a) of the main result we need to investigate the ``size" of the zero set of $\theta$. We do this in the following lemma.  Only item (\ref{item:meas Gamma}) is used in the proof of Theorem \ref{main result}; we include item (\ref{item: H dim Gamma}) in order to provide a better description of the zero set of $\theta$.

Throughout this section we use $|A|$ to denote the Lebesgue measure of $A\subset \rr^n$ and, for $t>0$,  $\Gamma_t:=\{x|\theta(x,t)=0\}$. 

\begin{lem}
\label{lem:Gammat}
Assume the hypotheses of Theorem \ref{main result}. 
\begin{enumerate}
\item \label{item:meas Gamma} Let $Q$ be a compact subset of $\rr^n$ and let $\ep>0$. There exists an open set $A_\ep$ such that $|A_\ep|\leq \ep$ and $Q\cap \Gamma_t\subset A_\ep$. 
\item \label{item: Gamma mea 0} For any $t>0$ we have $|\Gamma_t|=0$.
\item \label{item: H dim Gamma} We have,
\[
dim_H(\Gamma_t)\leq \exp(Nt) dim_H(\Gamma_0),
\]
where $dim_H$ is Hausdorff dimension and $N$ is the constant from assumption \ref{hyp:logL}. 
\end{enumerate}
\end{lem}
\begin{proof}
We provide only a sketch. Due to the characterization of $\theta$ provided by Theorem \ref{thm:u and ODE}, $\Gamma_t$ is exactly the image of $\Gamma_0$ under the map $\Phi_t$ defined in Theorem \ref{thm:u and ODE}. We recall that $\Phi_t$ is Holder continuous. Our assumption that $\Omega_0$ is compact implies that locally $\bdry\Omega_0$ is a graph of a uniformly continuous function. Together these two facts imply the claim of item (\ref{item:meas Gamma}) by a standard real analysis argument.

Item \ref{item:meas Gamma} implies that $|Q\cap \Gamma_t|=0$ for any compact set $Q$. Since there exists a countable cover of $\rr^n$ by compact sets, we find that $|\Gamma_t|=0$, as desired.

The definition and basic properties of Hausdorff dimension may be found in \cite{EvansGariepy}. In particular, if $f:\rr^n\rightarrow\rr^n$ is Holder with exponent $\alpha$, and $E\subset \rr^n$, then we have,
\[
dim_H(f(E))\leq \frac{1}{\alpha}dim_H(E).
\]
Item (\ref{item: H dim Gamma}) follows from this and the fact that $\Phi_t$ is Holder continuous with exponent $\exp(-Nt)$. 
\end{proof}
\begin{rem}
\label{rem:dim Gamma}
The estimate in item (\ref{item: H dim Gamma}) is fairly weak -- in particular, for times $t$ larger than $N^{-1}\ln(n/\dim_H(\Gamma_0))$, the estimate says only $\dim_H(\Gamma_t)\leq n$, which holds trivially. However, as far as we can tell, this is the best that we can do, because the map $\Phi_t$ is only Holder continuous, but not (necessarily) Lipschitz.
\end{rem}

\begin{proof}[Proof of part (a) of Theorem \ref{main result}]
Let $Q\subset\subset Q'\subset\subset \rr^n$ and let $p>0$. 
  Let us use $Z_k(x,t)$ to denote,
\[
Z_k(x,t)=W_\infty(x,t)-W_k(x,t).
\]
According to \cite[equations (15),(16)]{PV}, we have that $W_\infty$ satisfies,
\[
-\Delta W_\infty +W_\infty = p_\infty,
\]
where $p_\infty=(Id-G)^{-1}(W_\infty)\chi_{p_\infty>0}$.   Subtracting the equation that $W_k$ satisfies (\ref{eq:fromPandV}) from the previous line, we find that $Z_k$ satisfies,
\[
-\Delta Z_k +Z_k=p_\infty-p_k.
\]
Thus, standard estimates  for elliptic equations (see, for example, \cite[Theorem 9.11]{GT}) yield,
\[
||Z_k||_{W^{2,p}(Q)}\leq C (||p_\infty-p_k||_{L^p(Q')}+ ||Z_k||_{L^p(Q')})
\]
where the constant $C$ depends on $Q$, $Q'$, $n$ and $p$. 

Let $\ep>0$. According to Lemma \ref{lem:results perthame} the $W_k$ converge to $W_\infty$ locally uniformly. Thus we have that for $k$ large enough, the last term on the right-hand side of the previous line is bounded from above by $\ep/2C$.

According to Lemma \ref{lem:Gammat}, there exists an open set $A$ such that $ Q'\cap\Gamma_t\subset A$ and $|A|^p\leq \ep/8CP_M$.  Let us use $Q_1$, $Q_2$, and $Q_3$ to denote,
\[
Q_1=Q'\cap \{\theta>0\}\cap A^c, \  \  Q_2=Q'\cap \{\theta<0\}\cap A^c,\  \   Q_3=Q'\cap A.
\]
Thus we have,
\[
||p_\infty-p_k||_{L^p(Q')}\leq \sum_{i=1,2,3}||p_\infty-p_k||_{L^p(Q_i)}.
\]
According to Theorem \ref{main result}, we have that $p_k\rightarrow p_\infty$ locally uniformly on compact subsets of either $\{\theta>0\}$ or $\{\theta<0\}$. Since $Q_1$ and $Q_2$ are compact subsets of $\{\theta>0\}$ and $\{\theta<0\}$, respectively, we find that for $k$ large enough and $i=1,2$, $||p_\infty-p_k||_{L^p(Q_i)}\leq \ep/8C$. In addition, we have $||p_k||_{L^\infty(\rr^n)}\leq P_M$, so we find, 
\[
||p_\infty-p_k||_{L^p(Q_3)}\leq |A|^p2|P_M|\leq \frac{\ep}{4C}.
\]
Putting everything together yields
\[
||Z_k||_{W^{2,p}(Q)}\leq \ep
\]
for $k$ large enough, as desired.

\end{proof}

\appendix

\section{}
\label{second appendix}
In this appendix we establish the comparison result Theorem \ref{thm:comparison}. First we state and prove some preliminary lemmas. Throughout, we use $H$ to denote,
\[
H(x,t,p)=V(x,t)\cdot p.
\]
For $\ep>0$ let $\rho_\ep$ be a standard molifier. We use $H_\ep$ to denote the time-convolution of $H$  with $\rho_\ep$, defined as follows: Extend $H$ for $t\leq 0$ and $t\geq T$ by zero and define,
\[
H_\ep (x,t,p)=\int_\rr \rho_\ep(t-s)H(x,s,p)\, ds.
\]

The following is a rephrasing of Lemma 8.1 of Ishii. The proof works almost verbatim and we do not repeat it.
\begin{lem}
\label{lem:Hep}
Suppose $V$ satisfies hypotheses (\ref{bdsV}), (\ref{hyp:logL}) and (\ref{hyp:L1}). Let $K$ be a  compact subset  of $\rr^n \times \rr^n$. We have: 
\begin{enumerate}
\item $|H(x,t,p)-H(y,t,p)|\leq \sigma(|x-y|)|p|$.
\item Given $(x_0,p_0)$ and $\delta>0$, we set $b_\ep(t)$ to be,
\[
b_\ep(t)=\sup_{(x,p)\in B_\delta(x_0,p_0)}|H_\ep(x,t,p)-H(x,t,p)|.
\]
We have $\int_0^T b_\ep(t)\, dt\rightarrow 0$ as $\ep\rightarrow 0$.
\item \label{item L1CK} $H_\ep\rightarrow H$ in $L^1((0,T);C(K))$; i.e., for any $(x,p)\in \rr^n\times \rr^n$ we have,
\[
\int_0^T \sup_{(x,p)\in K}|H_\ep(t,x,p)-H(t,x,p)|\, dt \rightarrow 0 \text{ as }\ep\rightarrow 0.
\]
\end{enumerate}
\end{lem}

The next lemma is the analogy of \cite[Lemma 8.2]{Ishii}. The two differences are the regularity of $H$ in $x$ and the fact that now $u$ is USC and $v$ is LSC, while in \cite[Lemma 8.2]{Ishii} they are both continuous. The proof is almost identical to that of \cite[Lemma 8.2]{Ishii}, and we use much of the same notation. We provide the proof for the sake of completeness. 
\begin{lem}
\label{lem:subsolution}
Suppose $V$ satisfies hypotheses (\ref{bdsV}), (\ref{hyp:logL}) and (\ref{hyp:L1}). Let $u$, $v$ be, respectively, an USC sub-solution and a LSC supersolution of (\ref{eq:main}). Let $Q$ be an open subset of $\rr^n\times\rr^n\times (0,T)$ and let $\bar{H}:\rr^n\times \rr^n\times (0,T)\times \rr^n\times \rr^n \rightarrow \rr$ be defined by,
\[
\bar{H}(x,y,t,p,q)=-\sigma(|x-y|)|p|-M|p+q|.
\]
 Then $w(x,y,t)$ defined by
\[
w(x,y,t)=u(x,t)-v(y,t)
\]
is a viscosity subsolution (in the classical sense) of,
\[
w_t+\bar{H}(x,y,Dw)\leq 0
\]
in $Q$.
\end{lem}
We will need the following auxiliary lemma:
\begin{lem}
\label{lem:auxsubsoln}
Suppose $W\in USC(\rr^{2n}\times (0,T)^2)$, $\xi_\ep\in C((0,T)^2)$ and $\xi_\ep(t,s)$ converges to zero uniformly. Suppose 
\[
(x,y,t)\mapsto W(x,y,t,t)
\] 
has a local maximum at $(x_0,y_0,t_0)$ which is strict in $B_\gamma((x_0,y_0,t_0))$ for some $\gamma>0$. Let $\Phi_{\ep,\alpha}$ be given by,
\[
\Phi_{\ep,\alpha} (x,y,t,s)=W(x,y,t,s)+\xi_\ep(t,s)-\frac{(t-s)^2}{\alpha},
\]
and let  $(x_{\ep,\alpha}, y_{\ep,\alpha}, t_{\ep,\alpha}, s_{\ep,\alpha})$ be the maximum of $\Psi$ on $B_\gamma(x_0,y_0,t_0)$. Then
\begin{equation}
\label{limpep}
\lim_{\ep\rightarrow 0}\lim_{\alpha\rightarrow 0} |(x_{\ep,\alpha}, y_{\ep,\alpha}, t_{\ep,\alpha}, s_{\ep,\alpha})-(x_0,y_0,t_0, t_0)|=0
\end{equation}
and, for each $\ep$,
\begin{equation}
\label{talpha=salpha}
\lim_{\alpha\rightarrow 0}\frac{|t_{\ep,\alpha}-s_{\ep,\alpha}|}{\alpha}=0.
\end{equation}
\end{lem}
\begin{proof}[Proof of Lemma \ref{lem:auxsubsoln}]
Suppose $(x_{\ep,\alpha}, y_{\ep,\alpha}, t_{\ep,\alpha}, s_{\ep,\alpha})$ has a subsequential limit $(\bar{x}_\ep,\bar{y}_\ep, \bar{t}_\ep,\bar{s}_\ep)\in \bar{B}_\gamma(x_0,y_0,t_0)$ as $\alpha\rightarrow 0$. According to an argument standard to viscosity theory (see, for example, Lemma 4.1 of Crandall \cite{C}), we have 
\begin{equation}
\label{tbar=sbar}
\lim_{\alpha\rightarrow 0}\frac{|t_{\ep,\alpha}- s_{\ep,\alpha}|}{\alpha}=0,
\end{equation}
(so in particular $\bar{t}_\ep=\bar{s}_\ep$), and $(\bar{x}_\ep,\bar{y}_\ep, \bar{t}_\ep,\bar{t}_\ep)$ is a local maximum of 
\[
(x,y,t,t)\mapsto W(x,y,t,t)-\xi_\ep(t,t).
\]
Now let us suppose $(\bar{x}_\ep,\bar{y}_\ep, \bar{t}_\ep)$ has a subsequential limit $(\bar{x}, \bar{y}, \bar{t})\in \bar{B}_\gamma(x_0,y_0,t_0)$. Then, because $(\bar{x}_\ep,\bar{y}_\ep, \bar{t}_\ep,\bar{t}_\ep)$ is a local maximum of the above function, we find,
\[
W(x_0, y_0,t_0,t_0)-\xi_\ep(t_0,t_0) \leq W(\bar{x}_\ep,\bar{y}_\ep, \bar{t}_\ep,\bar{t}_\ep)-\xi_\ep( \bar{t}_\ep,\bar{t}_\ep).
\]
Let us now take $\limsup_{\ep\rightarrow 0}$. We use that $W$ is USC and the fact that $\xi_\ep$ converges uniformly to zero to find,
\[
W(x_0, y_0,t_0,t_0)\leq W(\bar{x}, \bar{y}, \bar{t}, \bar{t}).
\]
But, by assumption, $(x_0, y_0,t_0,t_0)$ was a strict local maximum, hence we must have
\[ 
(x_0, y_0,t_0,t_0)=(\bar{x}, \bar{y}, \bar{t}, \bar{t}).
\]
This establishes (\ref{limpep}). 
\end{proof}

\begin{proof}[Proof of Lemma \ref{lem:subsolution}]
\textbf{Step 1}. Let $H$, $\bar{H}$ and $w$ be as in the statement of the lemma. Our assumptions on $V$ imply,
\begin{equation}
\label{barHH}
\bar{H}(x,y,t,p, q)\leq H(x,t,p)-H(y,t,-q).
\end{equation} 
Let $\phi\in C^1(Q)$ and $(x_0,y_0, t_0)\in Q$. We denote,
\[
X_0=(x_0,y_0), p_0=D_x\phi(X_0,t_0), q_0=D_y\phi(X_0,t_0),
\]
and assume 
\[
(x,y,t)\mapsto w(x,y,t)+\int_0^tb(s)\, ds - \phi(x,y,t)
\]
has a  local maximum at $(x_0,y_0,t_0)$ which is strict in $B_\gamma(x_0,y_0,t_0)$. We aim to establish,
\[
\phi_t(x_0,y_0,t_0)+\bar{H}(x_0,y_0, p_0,q_0)\leq 0.
\]

\textbf{Step 2.} Let $\delta>0$. Let $H_\ep$ and $b_\ep$ be as in Lemma \ref{lem:Hep}, in particular,
\begin{equation}
\label{bedefn}
b_\ep(t)=\sup_{(x,p)\in B_\delta(x_0,p_0)}|H_\ep(x,t,p)-H(x,t,p)|.
\end{equation}

Let $(x_{\ep,\alpha}, y_{\ep,\alpha}, t_{\ep,\alpha},  s_{\ep,\alpha})$ be the maximum of 
\[
u(x,t)-v(y,t) -\int_0^t b_\ep(r)\, dr - \int_0^s -b_\ep(r)\, dr -\phi(x,y,t)-\frac{(t-s)^2}{\alpha}.
\]
 on $B_\gamma(x_0,y_0,t_0)$. We apply Lemma \ref{lem:auxsubsoln}  with 
\[
\xi_\ep(t,s)= -\int_0^t b_\ep(r)\, dr - \int_0^s b_\ep(r)\, dr
\]
to find,
\[
\lim_{\ep\rightarrow 0}\lim_{\alpha\rightarrow 0} |(x_{\ep,\alpha}, y_{\ep,\alpha}, t_{\ep,\alpha}, s_{\ep,\alpha})-(x_0,y_0,t_0, t_0)|=0.
\]
(We show why $\xi_\ep$ satisfies the hypotheses of Lemma \ref{lem:auxsubsoln}. We have, 
\[
\sup_{t,s}|\xi_\ep(t,s)|\leq \sup_{t,s}(\int_0^t|b_\ep(r)|\, dr+\int_0^s|b_\ep(r)|\, dr)\leq 2\int_0^1|b_\ep(r)|\, dr,
\]
and according to   Lemma \ref{lem:Hep},  $b_\ep$ converges to $0$ in $L^1(0,T)$. Therefore $\xi_\ep$ converges uniformly to zero and the hypotheses of Lemma \ref{lem:auxsubsoln} are satisfied.)

\textbf{Step 3.} 
 Let us use $p_{\ep,\alpha}$ to denote $D_x\phi(x_{\ep,\alpha}, t_{\ep,\alpha})$ and  $q_{\ep,\alpha}$ to denote $D_y\phi(y_{\ep,\alpha}, s_{\ep,\alpha})$. We claim that, for $\ep$ and $\alpha$ small enough, we have,
\begin{equation}
\label{H-}
(H_\ep, b_\ep)\in H^-(x_{\ep,\alpha}, t_{\ep,\alpha}, p_{\ep,\alpha})
\end{equation}
and
\begin{equation}
\label{H+}
((y,s,q)\mapsto (H_\ep(x_{\ep,\alpha}, s, p_{\ep,\alpha}) - \bar{H}(x_{\ep,\alpha}, y,  p_{\ep,\alpha}, -q), b_\ep(s)))\in H^+(x_{\ep,\alpha}, s_{\ep,\alpha},p_{\ep,\alpha}).
\end{equation}
Indeed, we first observe that, since 
\[
(x_{\ep,\alpha}, y_{\ep,\alpha}, t_{\ep,\alpha}, s_{\ep,\alpha})\rightarrow(x_0,y_0,t_0, t_0)
\] 
and $\phi$ is continuous, we have 
\[
(p_{\ep,\alpha}, q_{\ep,\alpha})\rightarrow (p,q).
\]
Therefore, for $\ep$ and $\alpha$ small enough, 
\begin{equation}
\label{BepB0}
B_{\delta/2}(x_{\ep,\alpha}, y_{\ep,\alpha}, t_{\ep,\alpha}, s_{\ep,\alpha}, p_{\ep,\alpha}, q_{\ep,\alpha})\subset B_\delta((x_0,y_0,t_0, t_0, p_0,q_0)).
\end{equation}

According to the definition of $b_\ep$ (line (\ref{bedefn})), we have,
\[
H_\ep(x,t,p)-b_\ep(t)\leq H(x,t,p)
\]
for all $0\leq t\leq T$ and for $(x,p)\in B_\delta(x_0,p_0)$. By (\ref{BepB0}), we have $B_{\delta/2}(x_{\ep,\alpha}, p_{\ep,\alpha})\subset B_\delta(x_0,p_0)$, hence the previous line holds on $B_{\delta/2}(x_{\ep,\alpha}, p_{\ep,\alpha})$ and for $0\leq t\leq T$. This implies that (\ref{H-}) holds.

We now establish (\ref{H+}). For this we recall (\ref{barHH}), which reads, with $-q$ instead of $q$, $s$ instead of $t$,
\[
\bar{H}(x_{\ep,\alpha},y,p_{\ep,\alpha},-q)\leq H(x_{\ep,\alpha},s,p_{\ep,\alpha})-H(y,s,q).
\]
Rearranging we find,
\[
H(y,s,q)\leq H(x_{\ep,\alpha},s,p_{\ep,\alpha}) - \bar{H}(x_{\ep,\alpha},y,p_{\ep,\alpha},-q)-b(t).
\]
 We use the definition of $b_\ep$ to bound from above the first term on the right-hand side of the previous line and find,
\[
H(y,s,q)\leq H_\ep(x_{\ep,\alpha},s,p_{\ep,\alpha}) +b_\ep(s) - \bar{H}(x_{\ep,\alpha},y,p_{\ep,\alpha},-q)
\]
on $B_\delta(x_{\ep,\alpha}, p_{\ep,\alpha})$, and hence on $B_{\delta/2}(y_{\ep,\alpha}, -q_{\ep,\alpha})$ and for   all $t$.

\textbf{Step 4.} Since $(x_{\ep,\alpha}, y_{\ep,\alpha}, t_{\ep,\alpha},  s_{\ep,\alpha})$ is a maximum of $\Psi_{\ep,\alpha}$, we have that $(x_{\ep,\alpha}, t_{\ep,\alpha})$ is  a local maximum of 
\[
(x,t)\mapsto u(x,t)-\left(\int_0^tb_\ep(r)\,dr +\psi(x,y_{\ep,\alpha}, t) +\frac{(t-s_{\ep,\alpha})^2}{\alpha}\right),
\]
and $(y_{\ep,\alpha}, s_{\ep,\alpha})$ is a local minimum of,
\[
(y,s)\mapsto v(y,s)-\left(\int_0^s(b(r)-b_\ep(r))\, dr-\phi(x_{\ep,\alpha}, y, t_{\ep,\alpha}) - \frac{(t_{\ep,\alpha}-s)^2}{\alpha}\right).
\]
Together with (\ref{H-}) and (\ref{H+}), and the fact that $u$ and $v$ are, respecively, a sub- and a super-solution, this implies:
\[
\phi_t(x_{\ep,\alpha}, y_{\ep,\alpha}, t_{\ep,\alpha}) +2\frac{t_{\ep,\alpha}-s_{\ep,\alpha}}{\alpha} + H_\ep(x_{\ep,\alpha}, s_{\ep,\alpha}, D_x\phi(x_{\ep,\alpha}, y_{\ep,\alpha}, t_{\ep,\alpha}))\leq 0,
\]
and
\[
- 2\frac{t_{\ep,\alpha}-s_{\ep,\alpha}}{\alpha} + H_\ep(x_{\ep,\alpha}, t_{\ep,\alpha}, p_{\ep,\alpha}) - \bar{H}(y_{\ep,\alpha}, x_{\ep,\alpha},  p_{\ep,\alpha}, D_y\phi(y_{\ep,\alpha}, x_{\ep,\alpha}, t_{\ep,\alpha}))\geq 0.
\]
Subtracting the previous line from the one before, and rearranging gives,
\begin{align*}
\phi_t(x_{\ep,\alpha}, y_{\ep,\alpha}, t_{\ep,\alpha}) &+4\frac{t_{\ep,\alpha}-s_{\ep,\alpha}}{\alpha} +\bar{H}(x_{\ep,\alpha}, y_{\ep,\alpha},  p_{\ep,\alpha}, D_y\phi(y_{\ep,\alpha}, x_{\ep,\alpha}, t_{\ep,\alpha})) 
\\&\leq  H_\ep(x_{\ep,\alpha}, s_{\ep,\alpha}, p_{\ep,\alpha}) - H_\ep(x_{\ep,\alpha}, t_{\ep,\alpha}, D_x\phi(x_{\ep,\alpha}, y_{\ep,\alpha}, t_{\ep,\alpha})).
\end{align*}
Let us recall that $p_{\ep,\alpha}=D_x\phi(x_{\ep,\alpha}, y_{\ep,\alpha}, t_{\ep,\alpha})$. Let us also recall that, according to Lemma \ref{lem:auxsubsoln}, we have $
|t_{\ep,\alpha}-s_{\ep,\alpha}|\rightarrow 0$ as $\alpha\rightarrow 0$.
 
We now take the limit as $\alpha\rightarrow 0$ of the previous line. The second term on the left-hand side converges to zero. In addition, because, for fixed $\ep$, $H_\ep$ is continuous in $t$, we find that the righthand side converges to zero. Thus we obtain,
\[
\lim_{\alpha\rightarrow 0 }(\phi_t(x_{\ep,\alpha}, y_{\ep,\alpha}, t_{\ep,\alpha}) 
+\bar{H}(x_{\ep,\alpha}, y_{\ep,\alpha}, p_{\ep,\alpha}, D_y\phi(y_{\ep,\alpha}, x_{\ep,\alpha}, t_{\ep,\alpha})))\leq 0.
\]
Now taking the limit $\ep\rightarrow 0$, using (\ref{limpep}) of Lemma \ref{lem:auxsubsoln} and the fact that $\phi$ and $\bar{H}$ are continuous yields,
\[
\phi_t(x_0,y_0,t_0)+\bar{H}(x_0,y_0,D_x\phi(x_0,y_0,t_0), D_y\phi(x_0,y_0,t_0))\leq 0,
\]
as desired.

\end{proof}

The remainder of the proof of the theorem, including the set-up that is presented here and the presentation, closely follows Stromberg \cite{S}. In fact, the ideas we use from there can also be seen in the earlier work \cite{CIL} of Crandall, Ishii and  Lions.

For the remainder of this section we take $r\in [0, \infty)$ and $\alpha\in (0, \infty)$. Let us use $\theta$ to denote,
\[
\theta(r)=
\begin{cases}
0 \text{ if }r=0,\\
\infty \text{ otherwise}.
\end{cases}
\]
We will define $\theta_\alpha(r)$ that converges to $\theta(r)$ as $\theta\rightarrow \infty$ for all $r$. 
We first define $G_\alpha$ by,
\[
G_\alpha(r)=\frac{1}{\alpha}\exp(-\int_r^1\frac{1}{\sigma(s)}\, ds ),
\]
and then,
\[
\theta_\alpha(r)= G_\alpha((\alpha^2+r^2)^{1/2}). 
\]
For future use we compute, for $r\in [0,1]$,
\[
\int_r^1\frac{1}{\sigma(s)}\, ds = \int_r^1\frac{1}{Ns|\ln(s)|}\, ds = -\frac{1}{N}\int_{|\ln(r)|}^0 \frac{1}{u}\, du = \frac{1}{N}[\ln(u)]_0^{|\ln(r)|} = \frac{1}{N}\ln(|\ln (r)|),
\]
so that 
\[
G_\alpha(r)=\frac{1}{\alpha}\exp(-\frac{1}{N}\ln(|\ln (r)|) ) = \frac{e^{-\frac{1}{N}}}{\alpha} \frac{1}{|\ln r|}.
\]
According to the definition of $\theta_\alpha$ we have,
\begin{equation}
\label{theta explicit}
\theta_\alpha(r)=    G_\alpha((\alpha^2+r^2)^{1/2})=   \frac{e^{-\frac{1}{N}}}{\alpha} \frac{1}{|\ln (\alpha^2+r^2)^{1/2}|}.
\end{equation}

We summarize some  properties of $\theta_\alpha$ in the following lemma.
\begin{lem}
\label{lem:properties theta}
Let $\theta_\alpha$ be defined as above. We have,
\begin{enumerate}
\item \label{item:basic theta} $\theta_\alpha$ is non-negative, increasing and smooth on $[0, \infty)$ for each $\alpha>0$,
\item \label{item: derivs}$D_x\theta_\alpha(|x-y|)=-D_y\theta_\alpha(|x-y|)$, 
\item \label{item: deriv bd}
\begin{equation}
\label{Dtheta}
|D_x\theta_\alpha(|x-y|)|\leq \frac{\theta_\alpha(|x-y|)}{\sigma(|x-y|)},
\end{equation}
\item \label{item:pointwise}
$\theta_\alpha(r)\rightarrow \theta(r)$ for all $r$ as $\alpha\rightarrow 0$, 
\item \label{item:otherconv} if $r_\alpha\rightarrow r$, then
\begin{equation}
\label{liminfthetaalpha}
\theta(r)\leq \liminf_{\alpha\rightarrow 0}\theta_\alpha(r_\alpha).
\end{equation}
\end{enumerate}
 \end{lem}

\begin{proof} 
Item (\ref{item:basic theta}) holds because  $G_\alpha$ is non-negative, increasing and smooth on $[0, \infty)$. 

Item (\ref{item: derivs}) follows directly from the definition of $\theta_\alpha$ and item (\ref{item:basic theta}).
  
To establish item (\ref{item: deriv bd}), we compute,
\[
G_\alpha'(r)=\frac{G_\alpha(r)}{\sigma(r)},
\]
and so,
\begin{align*}
\theta_\alpha'(r) &= G_\alpha'((\alpha^2+r^2)^{1/2})\frac{r}{(\alpha^2+r^2)^{1/2}}=\frac{G_\alpha((\alpha^2+r^2)^{1/2})}{\sigma((\alpha^2+r^2)^{1/2})}\frac{r}{(\alpha^2+r^2)^{1/2}}.
\end{align*}
Taking absolute value yields,
\begin{equation*}
|\theta_\alpha'(r) |\leq \frac{G_\alpha((\alpha^2+r^2)^{1/2})}{\sigma((\alpha^2+r^2)^{1/2})}=\frac{\theta_\alpha(r)}{\sigma((\alpha^2+r^2)^{1/2})},
\end{equation*}
where the equality follows from the definition of $\theta$. Since $\sigma$ is increasing, we have $\sigma((\alpha^2+r^2)^{1/2})\geq \sigma(r)$, and hence we find,
\begin{equation}
\label{theta prime r}
|\theta_\alpha'(r) |\leq \frac{\theta_\alpha(r)}{\sigma(r)}.
\end{equation}
Now let us consider $\theta_\alpha(|x-y|)$ for $x,y\in \rr^n$. We have,
\[
D_x\theta_\alpha(|x-y|)= \theta_\alpha'(|x-y|)\frac{(x-y)}{|x-y|},
\]
so taking absolute value and then applying (\ref{theta prime r}) with $r=|x-y|$ yields item (\ref{item: deriv bd}).

To establish item (\ref{item:pointwise}), we take $r=0$ in (\ref{theta explicit}) and 
then take the limit $\alpha\rightarrow \infty$, to find,
\[
\lim_{\alpha\rightarrow 0} \theta_\alpha(0)=  \lim_{\alpha\rightarrow 0} \frac{1}{\alpha} \frac{1}{|\ln \alpha|}=0.
\]
If $r\neq 0$, then, again using (\ref{item:pointwise}), we find,
\[
\lim_{\alpha\rightarrow 0} \theta_\alpha(r)= \lim_{\alpha\rightarrow 0} \frac{e^{-\frac{1}{N}}}{\alpha} \frac{1}{|\ln (\alpha^2+r^2)^{1/2}|} =\infty.
\]
This means that $\theta_\alpha$ converges pointwise to $\theta$, as desired.

To establish item (\ref{item:otherconv}), we consider the cases $r=0$ and $r\neq 0$. First, let us suppose $r=0$. Then (\ref{liminfthetaalpha}) holds because its left-hand side is zero, and we already know that the $\theta_\alpha$ are non-negative. Now let us suppose $r\neq 0$. Then for all $\alpha$ small enough we have $r_\alpha\geq r/2$. Since $\theta_\alpha$ is increasing for all $\alpha$ we find,
\[
\theta_\alpha(r_\alpha)\geq \theta_\alpha(r/2).
\]
Taking $\liminf$ yields,
\[
\liminf_{\alpha\rightarrow 0}\theta_\alpha(r_\alpha)\geq \liminf_{\alpha\rightarrow 0} \theta_\alpha(r/2)  = \infty = \theta(r),
\]
where the equalities follow from the first part of the lemma.
\end{proof}

We will abuse notation slightly and use $\theta(x,y,t)$ as well as $\theta(x,y)$ to mean $\theta(|x-y|)$, and similarly for $\theta_\alpha$.
\begin{lem}
\label{lem:penalize}
Let $K$ be a compact subset of $\rr^n$ and let $W$ be upper-semicontinuous and bounded on  $K\times K\times [0,T]$.  Let $\Psi_\alpha$ denote,
\[
\Psi_\alpha(x,y,t)=W(x,y,t)-\theta_\alpha(|x-y|).
\]
Let $(x_\alpha, y_\alpha, t_\alpha)\in K\times K\times [0,T]$ be where the maximum of $\Psi_\alpha$ is achieved on $K\times K\times [0,T]$. Suppose that, along a subsequence, the $(x_\alpha, y_\alpha, t_\alpha)$ converge to some $(\bar{x}, \bar{y}, \bar{t})$ as $\alpha\rightarrow 0$. Then:
\begin{equation} \label{itembx=by}
\bar{x}=\bar{y},
\end{equation}
\begin{equation}
\label{item maxW=W(bar)}
\max_{x\in K,\, t\in [0,T]} W(x,x,t)= W(\bar{x}, \bar{x}, \bar{t}),
\end{equation} 
and,
\begin{equation}\label{itemlimtheta}
 \lim_{\alpha\rightarrow 0}\theta_\alpha(|x_\alpha-y_\alpha|)=0.
 \end{equation}

\end{lem}

\begin{proof}
We first establish 
\begin{equation}
\label{barxgeqxalpha}
(W-\theta)(\bar{x}, \bar{y}, \bar{t})\geq \limsup_{\alpha\rightarrow 0}(W-\theta_\alpha)(x_\alpha, y_\alpha, t_\alpha).
\end{equation} 
To this end, we observe,
\begin{equation}
\label{limsup+}
\limsup_{\alpha\rightarrow 0}(W-\theta_\alpha)(x_\alpha, y_\alpha, t_\alpha)
= \limsup_{\alpha\rightarrow 0}W(x_\alpha, y_\alpha, t_\alpha)- \liminf_{\alpha\rightarrow 0}\theta_\alpha(x_\alpha, y_\alpha, t_\alpha).
\end{equation}
Since $W$ is upper-semicontinuous, the first term on the right-hand side of the previous line is bounded from above by $W(\bar{x}, \bar{y}, \bar{t})$. To bound the second term, we use item (\ref{item:otherconv}) of Lemma \ref{lem:properties theta} with $r_\alpha=|x_\alpha- y_\alpha|$ and $r=|\bar{x}-\bar{y}|$ to find,
\[
\liminf_{\alpha\rightarrow 0}\theta_\alpha(x_\alpha, y_\alpha, t_\alpha)\geq \theta(\bar{x}, \bar{y}, \bar{t}). 
\]
Thus (\ref{limsup+}) implies,
\[
\limsup_{\alpha\rightarrow 0}(W-\theta_\alpha)(x_\alpha, y_\alpha, t_\alpha)\leq W(\bar{x}, \bar{y}, \bar{t}) -\theta(\bar{x}, \bar{y}, \bar{t}),
\]
which is exactly (\ref{barxgeqxalpha}).

We proceed with the proof of the lemma. Let us use $(\tilde{x}, \tilde{y}, \tilde{t})$ to denote the place where the maximum of $W-\theta$ is achieved on $K\times K\times [0,T]$. Since $(x_\alpha, y_\alpha, t_\alpha)$ is the maximum of $W-\theta_\alpha$, we find,
\[
(W-\theta_\alpha)(x_\alpha, y_\alpha, t_\alpha)\geq (W-\theta_\alpha)(\tilde{x}, \tilde{y}, \tilde{t}). 
\]
We take $\limsup_{\alpha\rightarrow 0}$ of both sides. Using item (\ref{item:pointwise}) of Lemma \ref{lem:properties theta} on the right-hand side of the previous line yields,
\[
\limsup_{\alpha\rightarrow 0} (W-\theta_\alpha)(x_\alpha, y_\alpha, t_\alpha) 
\geq (W-\theta)(\tilde{x}, \tilde{y}, \tilde{t}).
\]
The definition of $(\tilde{x}, \tilde{y}, \tilde{t})$ yields,
\begin{equation}
\label{limWthetaMaxWtheta}
\limsup_{\alpha\rightarrow 0} (W-\theta_\alpha)(x_\alpha, y_\alpha, t_\alpha) 
\geq \max_{K\times K\times [0,T]} (W-\theta) 
\end{equation}
We use the previous line to bound from below the right-hand side of (\ref{barxgeqxalpha}) and obtain,
\begin{equation}
\label{maxKK}
(W-\theta)(\bar{x}, \bar{y}, \bar{t})\geq \max_{K\times K\times [0,T]} (W-\theta).
\end{equation}

We now establish (\ref{itembx=by}). Indeed, if $\bar{x}\neq \bar{y}$, then $\theta(\bar{x}, \bar{y}, \bar{t})=\theta(|x-y|)=\infty$. Therefore, (\ref{maxKK}) implies,
\[
\max_{K\times K\times [0,T]} (W-\theta) \leq -\infty.
\]
However, since $\theta(0)=0$ we have,
\[
\max_{K\times K\times [0,T]}  (W-\theta)\geq \max_{x\in K, t\in [0,T]} (W-\theta)(x,x,t)
= \max_{x\in K, t\in [0,T]} W(x,x,t)>-\infty.
\]
Hence $\bar{x}=\bar{y}$ must hold.

Now we notice that the right-hand side of (\ref{maxKK}) is bounded from below by the maximum of the same quantity but for points $(x,x)\in K\times K$. And, we use that $\bar{x}=\bar{y}$ to rewrite the left-hand side. We find,
\[
(W-\theta)(\bar{x}, \bar{x}, \bar{t})=(W-\theta)(\bar{x}, \bar{y}, \bar{t})\geq \max_{K\times K\times [0,T]} (W-\theta) \geq \max_{x\in K,\, t\in [0,T]} (W-\theta)(x,x,t).
\]
But the right-hand side of the previous line is bounded from above by the left-hand side, and therefore equality must hold for all the items:
\begin{equation} 
\label{fouritems}
(W-\theta)(\bar{x}, \bar{x}, \bar{t})=(W-\theta)(\bar{x}, \bar{y}, \bar{t})= \max_{K\times K\times [0,T]} (W-\theta) = \max_{x\in K,\, t\in [0,T]} (W-\theta)(x,x,t).
\end{equation}
 Moreover, since $\theta(x,x,t)=0$ for all $x$ and $t$, we obtain (\ref{item maxW=W(bar)}).

We will now prove the last claim of the lemma, equation (\ref{itemlimtheta}). 
Line (\ref{limWthetaMaxWtheta}) provides a lower bound for the left-hand side of (\ref{limsup+}), yielding,
\[
\limsup_{\alpha\rightarrow 0} W(x_\alpha, y_\alpha, t_\alpha)-\liminf \theta_\alpha(x_\alpha, y_\alpha, t_\alpha)\geq \max_{K\times K\times [0,T]} (W-\theta).
\]
The last equality in (\ref{fouritems}) yields that the right-hand side of the previous line is exactly $\max_{K\times [0,T]}W(x,x,t)$. Rearranging yields,
\[
\liminf_{\alpha\rightarrow 0} \theta_\alpha(x_\alpha, y_\alpha, t_\alpha) \leq  \limsup_{\alpha\rightarrow 0} W(x_\alpha, y_\alpha, t_\alpha) - \max_{K\times [0,T]}W(x,x,t).
\] 
We now use that $W$ is upper-semicontinuous and that $\bar{x}=\bar{y}$ to find that the right-hand side of the previous line is non-negative, so that,
\[
\liminf_{\alpha\rightarrow 0} \theta_\alpha(x_\alpha, y_\alpha, t_\alpha) \leq 0.
\]
But since $\theta_\alpha$ is non-negative, we have 
\[
\limsup_{\alpha\rightarrow 0} \theta_\alpha(x_\alpha, y_\alpha, t_\alpha) \geq 0,
\]
which, together with the previous line yields (\ref{itemlimtheta}), completing the proof of the lemma.
\end{proof}

We now present:
\begin{proof}[Proof of Theorem \ref{thm:comparison}]
\textbf{Step 1.} In this step we reduce the situation from all of $\rr^n$ to compact subsets $S_\beta$, which we define as the sub-level sets of:
\[
f(x,t)=t+\frac{1}{M}(1+|x|^2)^{1/2}.
\]
Namely, for $\beta>0$ we define the subsets $S_\beta$ of $\rr^n$ as:
\[
S_\beta = \{(x,t)\in \rr^n\times [0,T]: \, f(x,t)\leq \beta\}.
\]
Note $Df=\frac{1}{M}\frac{x}{(1+|x|^2)^{1/2}}$, so that $|Df|<1/M$.

Notice that if $\beta_1\leq\beta_2$ then $S_{\beta_1}\subset S_{\beta_2}$, and hence,
\[
\max_{S_{\beta_1}}(u-v)^+\leq \max_{S_{\beta_2}}(u-v)^+.
\]
We define the function $g$ by,
\[
g(\beta)=\ln(2+\max_{S_{\beta}}(u-v)^+),
\]
and notice that $g$ is positive and non-decreasing. We also define, 
\[
\psi_\beta(x,t)=\exp(g(\beta)(1+f(x,t)-\beta)).
\]
We notice
\[
\lim_{\beta\rightarrow\infty} \psi_\beta(x,t)=0 
\]
for all $(x,t)\in \rr^n\times [0,T]$. We also find,
\begin{equation}
\label{eqpsi}
\partial_t\psi_\beta - M|D\psi_\beta|=\psi_\beta g(\beta) (f_t-M|Df|)>\psi_\beta g(\beta)(1-M\frac{1}{M})=0.
\end{equation}

\textbf{Step 2.} We seek to establish that for all $\beta$,
\[
(u-v-\psi_\beta)(x,t)\leq \sup_{\rr^n}(u-v)^+(x,0) \text{ for }(x,t)\in S_\beta.
\]
If this holds, then we may take the limit $\beta\rightarrow 0$ of the previous line, use that $\lim_{\beta\rightarrow\infty} \psi_\beta(x,t)=0$, and find that the desired inequality holds, completing the proof.

We proceed to establish the previous line by contradiction. Thus, assume that there exists $\beta>0$ such that the previous line fails. This implies that there exists $c>0$ with,
\[
\sup_{(x,t)\in S_\beta}(u(x,t)-v(x,t)-\psi_\beta(x,t)-ct)> \sup_{\rr^n}(u-v)^+(x,0).
\]
We now stop writing the subscript $\beta$, as it is fixed. We use $W(x,y,t)$ to denote,
\[
W(x,y,t)=u(x,t)-v(y,t)-\psi_\beta(y,t)-ct,
\]
so that the previous line becomes,
\begin{equation}
\label{supwithc}
\sup_{(x,t)\in S}W(x,x,t)>\sup_{\rr^n}(u-v)^+(x,0).
\end{equation}

Let us consider,
\[
\Phi_\alpha(x,y,t)=u(x,t)-v(y,t)-\psi_\beta(y,t)-ct- \theta_\alpha(|x-y|),
\]
where $\theta_\alpha$ is as defined in the previous subsection. 
Let $(x_\alpha, y_\alpha, t_\alpha)$ be a maximum of $\Phi_\alpha$ on $S\times S\times [0,T]$. Notice that this exists as $S$ is compact and $\Phi_\alpha$ is upper-semicontinuous.

\textbf{Step 3.} In this step, we establish, for all $\alpha$ small enough:
\begin{itemize}
\item   $t_\alpha>0$, and,
\item  $(x_\alpha, t_\alpha)$ and $(y_\alpha, t_\alpha)$ are contained in the interior of $S$. 
\end{itemize} 
Let us suppose that the first statement is not true. That means that there exists a subsequence $(x_{\alpha_j}, y_{\alpha_j}, t_{\alpha_j})$ converging to $(\bar{x}, \bar{y}, 0)$. We now apply Lemma \ref{lem:penalize}. We find that  $\bar{x}= \bar{y}$ and,
\[
\max_{x\in S,\, t\in [0,T]} W(x,x,t)= W(\bar{x}, \bar{x}, 0)=u(\bar{x},0)-v(\bar{x},0)-\psi(\bar{x},0)\leq u(\bar{x},0)-v(\bar{x},0),
\]
where the second equality follows from the definition of $W$. We now use (\ref{supwithc}) to bound the left-hand side of the previous line from below and find,
\[
\sup_{\rr^n}(u-v)^+(x,0)<u(\bar{x},0)-v(\bar{x},0),
\]
which is impossible. Therefore the first item must hold.

Now let us suppose the second item does not hold. This means that there exists a subsequence $(x_{\alpha_j}, y_{\alpha_j}, t_{\alpha_j})$ converging to $(\bar{x}, \bar{x}, \bar{t})$ with $(\bar{x}, \bar{t})\in \bdry S$ (the  $x_{\alpha_j}$ and $y_{\alpha_j}$ converge to the same point due to Lemma  \ref{lem:penalize}). The definition of $S$ as the sub level set of $f$ implies, 
\[
f(\bar{x}, \bar{t})=\beta.
\]
The definitions of $\psi$ and $g$ therefore imply,
\[
\psi(\bar{x}, \bar{t})=\exp(g(\beta)) = 2+\max_{S}(u-v)^+.
\]
Using this in the definition of $W$ yields,
\[
W(\bar{x}, \bar{x}, \bar{t})=u(\bar{x}, \bar{t})-v(\bar{x}, \bar{t})-\psi(\bar{x}, \bar{t})-c\bar{t}=u(\bar{x}, \bar{t})-v(\bar{x}, \bar{t})-(2+\max_{S}(u-v)^+)-c\bar{t} <0.
\]
We again apply Lemma  \ref{lem:penalize} and find,
\[
\max_{x\in S,\, t\in [0,T]} W(x,x,t)= W(\bar{x}, \bar{x}, \bar{t}).
\]
We had just shown that the right-hand side of the previous line is negative, so we find,
\[
\max_{x\in S,\, t\in [0,T]} W(x,x,t)<0,
\]
which contradicts (\ref{supwithc}), as the right hand side of (\ref{supwithc}) is non-negative.

\textbf{Step 4.} Let us recall that, according to Lemma \ref{lem:subsolution}, we have that $w(x,y,t)=u(x,t)-v(y,t)$ is a subsolution (in the  standard viscosity sense) to 
\[
w_t+\bar{H}(x,y,t,D_xw, D_yw)\leq 0,
\]
with
\[
\bar{H}(x,y,t,p,q)=-\sigma(|x-y|)|p|-M|p+q|.
\] 
The claim we established in step 3, together with the fact that $(x_\alpha, y_\alpha, t_\alpha)$ is a maximum of $\Phi_\alpha$ on $S\times S\times [0,T]$, therefore yields,
\[
\psi_t(y_\alpha,t_\alpha)+c+\bar{H}(x_\alpha,y_\alpha,t_\alpha,D_x\theta_\alpha(|x_\alpha-y_\alpha|), D_y\psi(y_\alpha,t_\alpha)+D_y\theta_\alpha(|x_\alpha-y_\alpha|))\leq 0.
\]
The definition of $\bar{H}$ thus gives,
\[
\psi_t(y_\alpha,t_\alpha)+c-\sigma(|x_\alpha-y_\alpha|)|D_x\theta_\alpha(|x_\alpha-y_\alpha|)|-M|D_y\psi(y_\alpha,t_\alpha)|\leq 0,
\]
where we have also used that $D_x\theta_\alpha(|x_\alpha-y_\alpha|)=-D_y\theta_\alpha(|x_\alpha-y_\alpha|)$. According to (\ref{eqpsi}), the sum of the first and last terms is non-negative, so we find,
\[
c-\sigma(|x_\alpha-y_\alpha|)|D_x\theta_\alpha(|x_\alpha-y_\alpha|)|\leq 0.
\]
We use (\ref{Dtheta}) to estimate the left-hand side from below and find,
\[
c- \sigma(|x_\alpha-y_\alpha|)\frac{\theta_\alpha(|x_\alpha-y_\alpha|)}{\sigma(|x_\alpha-y_\alpha|)}\leq 0,
\]
which becomes,
\[
c\leq \theta_\alpha(|x_\alpha-y_\alpha|).
\]
However, equation (\ref{itemlimtheta}) of Lemma \ref{lem:penalize} says that the limit as $\alpha$ goes to zero of the right hand side of the previous line is zero. This yields the desired contradiction, completing the proof.
\end{proof}

\section{}
\label{first appendix}
In this section we establish  Theorem \ref{thm:existence} Theorem \ref{thm:u and ODE}, and Theorem \ref{thm:discontinuous}.

If $V\in C^1$ satisfies (\ref{bdsV}) and (\ref{hyp:logL}), then classical results (which we summarize in the proposition below) assert that there exists a unique flow $X$ generated by $V$, and relates $X$ to  the equation,
\begin{equation}
\label{epeqn}
\begin{cases}
\partial_t u +V Du =0 \text{ in }\rr^n\times (0,T),\\
u(x,0)=u_0(x),
\end{cases}
\end{equation}
We remark on the notation: the flow $X$ has three arguments: the time variable $t$, as well as the starting time and point, $(x,s)$. 
\begin{prop}
\label{ODE existence}
Let $V\in C^1$ be continuous and satisfy  (\ref{bdsV}) and (\ref{hyp:logL}).
\begin{enumerate}
\item There exists a unique $X:\rr^n\times (0,T)\times (0,T)\rightarrow \rr^n$ satisfying,
\[
X(x,s,t)=x+\int_s^tV(X(x,s,r), r)\, dr 
\]
on $\rr^n\times (0,T)\times (0,T)$. 
\item The maps $x\mapsto X(x,0,t)$ and $x\mapsto X(x,t,0)$ are inverses.
\item Let $u_0\in C^1(\rr^n)$. Then $u$ defined on $\rr^n\times (0,T)$ by,
\[
u(x,t)=u_0(X(x,0,t))
\]
is a classical solution of (\ref{epeqn}). 
\end{enumerate}
\end{prop}
We refer the reader to  \cite[Chapter 1]{Crippa} for the proof, as well as more exposition and further references.

We now establish some regularity results for $u$, which depend only on the constants in (\ref{hyp:logL}) and (\ref{bdsV}). For this we need two lemmas about the regularity of the flow $X$. We phrase them in terms of its trajectories, $t\mapsto X(t,s,x)$, which we denote by $\gamma(t)$.

\begin{lem}
\label{lem:Xep estimate space}
Let $V$ be continuous and satisfy (\ref{bdsV}) and (\ref{hyp:logL}). Let $0\leq t_0\leq T$, $x_1,x_2\in \rr^n$, and for $i=1,2$, let $\gamma_i$ solve the ODEs,
\[
\begin{cases}
\dot{\gamma}_i(t) = V(\gamma_i(t),t) \text{ on }(t_0,T),\\
\gamma_i(t_0)=x_i.
\end{cases}
\]
Then for $t\in [0,T]$,
\[
|\gamma_1(t)-\gamma_2(t)|\leq |x_1-x_0|^{\exp(-NT)}.
\]
\end{lem}
\begin{proof}
We have,
\[
\gamma_1(t)-\gamma_2(t) = x_1-x_0+\int_{t_0}^t V(\gamma_1(s),s)-V(\gamma_2(s),s)\, ds. 
\]
Taking absolute value and using (\ref{hyp:logL}) yields, for all $t$,
\begin{equation}
\label{r and R}
|\gamma_1(t)-\gamma_2(t)|\leq |x_1-x_0|+\int_{t_0}^t \sigma(|\gamma_1(s)-\gamma_2(s)|)\, ds.
\end{equation}
Let us use $R(t)$ to denote, 
\[
R(t)=|x_1-x_0|+\int_{t_0}^t \sigma(|\gamma_1(s)-\gamma_2(s)|)\, ds.
\]
Using a generalized Gronwall's inequality and the explicit expression $\sigma(r)=Nr|\ln(r)|$ we may compute,
\[
R(t)\leq |x_1-x_0|^{\exp(-N(t-t_0))}\leq |x_1-x_0|^{\exp(-NT)}.
\]
Together with (\ref{r and R}), this implies that the desired claim holds.
\end{proof}

\begin{lem}
\label{lem:Xep estimate time}
Let $V$ be continuous and satisfy (\ref{hyp:logL}) and (\ref{bdsV}). Let $x\in\rr^n$, $0\leq t_1,t_2\leq T$ and let $\gamma_i$ solve,
\[
\begin{cases}
\dot{\gamma}_i(t)=V(\gamma_i(t), t)\\
\gamma(t_i)=x.
\end{cases}
\]
Let $0\leq t\leq t_1$. Then
\[
|\gamma_1(t)-\gamma_2(t)|\leq (M|t_1-t_2|)^{\exp(-NT)}.
\]
\end{lem}
\begin{proof}
Using the definition of the $\gamma_i$ and doing simple manipulations yields,
\begin{align*}
\gamma_1(t)-\gamma_2(t)&=\int_{t_1}^t V(\gamma_1(s),s)\, ds - \int_{t_2}^t V(\gamma_1(s),s)\, ds\\
&= \int^{t_1}_t V(\gamma_2(s),s)-V(\gamma_1(s),s)\, ds +\int^{t_2}_{t_1} V(\gamma_2(s),s)\, ds.
\end{align*}
Taking absolute value and using (\ref{bdsV}) and (\ref{hyp:logL}) yields,
\[
|\gamma_1(t)-\gamma_2(t)|\leq \int^{t_1}_t \sigma(|\gamma_1(s)-\gamma_2(s)|)\, ds +M|t_1-t_2|.
\]
We conclude just as in the proof of the previous lemma.
\end{proof}

We write the bounds in the previous two lemmas so explicitly to highlight the fact that they depend only on the constants $M$, $N$, and not on the $C^1$ norm of $V$.

\begin{prop}
\label{prop:regularity of u}
Assume the hypotheses of Proposition \ref{ODE existence} and let $u$ be as given there. Assume that there exists a modulus of continuity $\omega$ with,
\[
|u_0(x)-u_0(y)|\leq \omega(|x-y|).
\] We have $||u||_{L^\infty}\leq ||u_0||_{L^\infty}$ and $u$ is uniformly continuous, with modulus that depends only on $\omega$, $T$, $N$ and $M$.
\end{prop}
\begin{proof}
We have, by triangle inequality, the definition of $u$, and the fact that $u_0$ is uniformly continuous,
\begin{align*}
|u(x_1,t_1)-u(x_2,t_2)|&\leq |u(x_1,t_1)-u(x_1,t_2)|+|u(x_1,t_2)-u(x_2,t_2)|\\
&= |u_0(X(0,t_1,x_1))-u_0(X(0,t_2,x_1))|+|u_0(X(0,t_2,x_1))-u_0(X(0,t_2,x_2))|\\
&\leq \omega(|X(0,t_1,x_1)-X(0,t_2,x_1)|)+\omega(|X(0,t_2,x_1)-X(0,t_2,x_2)|).
\end{align*}
Applying Lemma \ref{lem:Xep estimate time} and Lemma \ref{lem:Xep estimate space} to the two terms on the right-hand side of the previous line implies that there exists a modulus 
 $\tilde{\omega}$ that depends only on $\omega$, $T$, $N$ and $M$ such that,
 \[
 |u(x_1,t_1)-u(x_2,t_2)|\leq \tilde{\omega}(|(x_1,t_1)-(x_2,t_2)|).
 \]
\end{proof}

We now present the proofs of Theorem \ref{thm:existence} and Theorem \ref{thm:u and ODE}. 
\begin{proof}[Proof of Theorem \ref{thm:existence} and Theorem \ref{thm:u and ODE}]
Let $\rho^\ep(x,t)$ and $\tilde{\rho}^\ep(x)$ be standard mollifiers with integral 1. We regularize $V$ by convolution in both $x$ and $t$:
\[
V^\ep(x,t)=\int V(x-y,t-s)\rho^\ep(y,s)\, dy\, ds.
\]
Since $V$ satisfies (\ref{bdsV}) and (\ref{hyp:logL}),  $V^\ep$ does as well. We also have $V^\ep\in C^1$, and hence $V^\ep$ satisfies the hypotheses of Proposition \ref{ODE existence}.  Let $X^\ep$ be the flow given by Proposition \ref{ODE existence}, so that $X^\ep$ satisfies,
\begin{equation}
\label{Xep}
X^\ep(x,s,t)=x+\int_s^tV^\ep(X^\ep(x,s,r),r)\, dr.
\end{equation}
We also regularize the initial data: we define,
\[
u_0^\ep(x)=\int_{\rr^n} u_0(x-y)\tilde{\rho}^\ep(y)\, dy,
\]
so that $u_0^\ep$ is $C^1$ and uniformly continuous with modulus independent of $\ep$. 

Let $u^\ep(x,t)=u_0^\ep(X^\ep(x,0,t))$, so that according to Proposition \ref{ODE existence}, $u^\ep$ satisfies, 
\[
\begin{cases}
\partial_t u^\ep +V^\ep Du^\ep =0 \text{ in }\rr^n\times (0,T),\\
u^\ep(x,0)=u_0^\ep(x).
\end{cases}
\]
According to Proposition \ref{prop:regularity of u}, the $u^\ep$ are uniformly continuous in $x$ and $t$ with modulus independent of $\ep$ and  $||u_\ep ||_{L^\infty}\leq ||u_0||_{L^\infty}<\infty$. By the Arzela-Ascoli Theorem there exists a subsequence which converges locally uniformly to some $u$. 

Let $K$ be a compact subset of $\rr^n\times (0,T)$. Item (\ref{item L1CK}) of Lemma \ref{lem:Hep}, together with an estimate similar to (\ref{Hep H}), imply that the $V^\ep$ converge to $V$ in $L^1((0,T), C(K))$ (indeed, Lemma \ref{lem:Hep} concerns regularizing via convolution in time, while (\ref{Hep H}) is an estimate regarding convolution in space). Thus we may apply Proposition \ref{prop:stability} and conclude that $u$ is a viscosity solution of (\ref{eq:main}). The stated regularity of $u$ is a consequence of the regularity of the $u^\ep$ given in Proposition \ref{prop:regularity of u}. 

We have thus completed the proof of Theorem \ref{thm:existence}, and continue with this setup to establish Theorem \ref{thm:u and ODE}. 

 According to Lemma \ref{lem:Xep estimate space} and Lemma \ref{lem:Xep estimate time}, for $s,t\in (0,T)$, the map,
\[
(x,s, t)\mapsto X^\ep(x,s,t)
\]
is Holder continuous, uniformly in $\ep$. The $L^\infty$ bound (\ref{bdsV})  on $V$ implies $|X^\ep(x,s,t)|\leq |x|+TM$ for $t\in (0,T)$. Thus the Arzela-Ascoli theorem implies that there exists a subsequence of $\ep_j\rightarrow 0$ as $j\rightarrow \infty$ and an $X(x,s,t)$ such that $X^{\ep_j}$ converge to $X$ locally uniformly on $(0,T)\times (0,T)\times \rr^n$. In addition, $X$ is therefore also Holder continuous.

Since $V^{\ep_j}$ converge to $V$ in $L^1((0,T), C(K))$ for any compact set $K$ and  satisfy (\ref{hyp:logL}) uniformly in $\ep_j$, we find that for any $s,t\in (0,T)$,
\[
\lim_{j\rightarrow\infty } \int_s^tV^{\ep_j}(X^{\ep_j}(x,s,r), r)\, dr = \int_s^tV(X(x,s,r), r)\, dr.
\]
Thus taking a (pointwise) limit of (\ref{Xep}) along the subsequence $\ep_j$ implies that $X$ satisfies (\ref{X and V}). 
 
Since $X^{\ep_j}$ converge to $X$ locally uniformly, we find $u_0^\ep(X^\ep(x,0,t))$ converges to $u_0(X(x,0,t))$. Thus,
\[
u(x,t)=\lim_{j\rightarrow \infty} u_0^{\ep_j}(X^{\ep_j}(x,0,t))=u_0(X(x,0,t)),
\]
as desired. 

The regularity of the map $\Phi$ follows from the regularity of $X$.

\end{proof}

\subsection{Proof of Theorem \ref{thm:discontinuous}}
We now have the ingredients needed to establish Theorem \ref{thm:discontinuous}. 
\begin{proof}
Let us first approximate the initial data $u_0$ from below by $u_0^{-,k}$ and from above $u_0^{+,k}$, where $u_0^{\pm,k}$ are  continuous functions from $\rr^n$ to $[0,1]$. We construct these functions such that $u_0^{\pm,k} = 1$ in $\Omega_0^{-,k}:=\{x: d(x, \Omega_0^C)\geq 1/k\}$ and $u_0^{\pm,k} = 0$ outside of $\Omega_0^{+,k}:= \{x: d(x, \Omega_0) \leq 1/k\}$. 

Next, let $u^{-,k}$ and $u^{+,k}$ be the corresponding unique viscosity solutions of (\ref{eq:main}) with aforementioned initial data $u_0^{-,k}$ and $u_0^{+,k}$. Due to Theorem 2.7 we have representation formulas for these solutions. In particular we have 
$$
u^{\pm,k} = 1 \hbox{ in } \Omega_t^{-,k}:= \{x: X(x,t,0)\in \Omega_0^{-,k}\}
$$
and
$$
u^{\pm, k} = 0 \hbox{ in } \Omega_t^{+,k}:= \{x: X(x,t,0)\in \Omega_0^{+,k}\}.
$$
The properties of the flow map $X$ summarized in Theorem \ref{thm:u and ODE}, including invertibility and continuity with respect to $x$ and $t$, imply,
\begin{equation}
\label{limsup}
\limsup_{k\to\infty}{ }^* u^{-,k} (\cdot,t)=\limsup_{k\to\infty}{ }^* u^{+,k} (\cdot,t)= \chi_{\bar{\Omega_t}} 
\end{equation}
and 
\begin{equation}
\label{liminf}
\liminf_{k\to\infty}{ }^* u^{-,k}(\cdot,t)=\liminf_{k\to\infty}{ }^* u^{+,k}(\cdot,t) = \chi_{\Omega_t}.
\end{equation}
 Due to the standard stability property of viscosity solutions, $\chi_{\bar{\Omega_t}} =u^*$ is a subsolution of \eqref{transport}, and $u^* =(u_0)^*$ due to the continuity of $X(x,t,0)$ with respect to $t$. Similarly we have  $u$ is a supersolution of \eqref{transport} with $u_* = (u_0)_*$. So we have at least one solution of \eqref{transport} with initial data $u_0$.

Now if there is any other solution $w$ of \eqref{transport} with initial data $u_0$, then by comparison principle $w^* \leq u^{+,k}$ and $w_* \geq u^{-,k}$.  Taking $\limsup$ and $\liminf$, repectively, of these inequalities, and then using (\ref{limsup}) and (\ref{liminf}) yields, 
\[
w^*\leq  \chi_{\bar{\Omega_t}}=u^*,
\]
and 
\[
w_*\geq \chi_{\Omega_t}=u.
\]
Putting the previous two lines together yields, 
\[
u\leq w_*\leq w\leq  w^*\leq u^*,
\]
 as desired.
\end{proof}

\section{}
\label{third appendix}

Throughout this section we use $H$ to denote, $H(x,s,p)=V(x,s)\cdot p$. 
First we provide:
\begin{proof}[Proof of Lemma \ref{lem:psiepsubsoln}]
We establish that $\psi_\ep$ is a subsolution. The proof that $\bar{\psi}_\ep$ is a supersolution is analogous.

Let us use $C_1$ to denote, $C_1=||\phi||_{C^3(\rr^n)}$. 
Let $\sigma$ be as in (\ref{hyp:logL}) and define $\tilde{\sigma}$ by,
\[
\tilde{\sigma}(\ep)=\int \rho(z)\sigma(\ep z)\, dz.
\]
We remark that $\tilde{\sigma}$ is a continuous decreasing function with $\lim_{\ep\rightarrow 0}\tilde{\sigma}(\ep)=0$. 
We take $\ep_1$ be such that 
\[
C_1\tilde{\sigma}(\ep_1)\leq \frac{\alpha}{4}.
\]
We have,
\begin{align*}
|H^\ep(x,t,p)-H(x,t,p)| &= |\int \rho_\ep(y) H(x-y,t,p)-H(x,t,p)\, dy| \\
&\leq \int \rho_\ep(y) |H(x-y,t,p)-H(x,t,p)|\, dy\\
&\leq \int \rho_\ep(y) |p|\sigma(y)\, dy =|p|\int \frac{1}{\ep^n}\rho\left(\frac{y}{\ep}\right)\sigma(y)\, dy = |p|\int \rho(z)\sigma(\ep z)\, dz,
\end{align*}
so that the definition of $\tilde{\sigma}$ implies,
\begin{equation}
\label{Hep H}
|H^\ep(x,t,p)-H(x,t,p)|\leq |p|\tilde{\sigma}(\ep).
\end{equation}
Let us now take $\ep=\ep_1$. The previous estimate, together with our choice of $\ep_1$, imply that for any $x\in \rr^n$,
\begin{equation}
\label{Hep1}
|H^{\ep_1}(x,t,D\phi(x))-H(x,t,D\phi(x))|\leq |D\phi(x)|\tilde{\sigma}(\ep_1)\leq \frac{\alpha}{4},
\end{equation}
which is exactly (\ref{H^ep minusH}). 
In addition, we have,
\begin{align*}
D_xH^\ep(x,t,p)&= D_x\int_y H(y,t,p)\frac{1}{\ep^n}\rho\left(\frac{x-y}{\ep}\right)\, dy=\int_y H(y,t,p)\frac{1}{\ep^{n+1}}D\rho\left(\frac{x-y}{\ep}\right)\, dy \leq \\
&\leq M |p| \int_y \frac{1}{\ep^{n+1}}D\rho\left(\frac{x-y}{\ep}\right)\, dy \leq \frac{M|p|}{\ep}C. 
\end{align*}

Now that we've fixed $\ep$, we take $\bar{h}$ to be,
\begin{equation}
\label{barh}
\bar{h} = \frac{\alpha}{16M(M C_1 + M C C_1 \ep^{-1})},
\end{equation}
where $C$ is as in the previous line. To establish that $\psi_\ep$ is a subsolution of (\ref{eq:main}) on $\rr^n\times (t,t+\bar{h})$, we suppose $\gamma\in C^1(\rr^n\times [0,T])$, $(x_0,r_0)\in \rr^n\times (t,t+\bar{h})$, $G,b\in H^-( x_0,r_0, D\gamma(x_0,r_0))$, and 
\begin{equation}
\label{fn psi gamma}
(x,r)\mapsto \psi_\ep(x,r) + \int_0^r b(s)\, ds -\gamma(x,r)
\end{equation}
has a local maximum at $(x_0,r_0)$. We want to show,
\[
\gamma_r(x_0,r_0)+G(x_0,r_0, D\gamma(x_0,r_0))\leq 0.
\]
We proceed by contradiction and suppose that the previous line does not hold. Thus there exists some $a>0$ such that,
\[
2a\leq \gamma_r(x_0,r_0)+G(x_0,r_0, D\gamma(x_0,r_0)).
\]
The continuity of $\gamma$, its derivatives, and $G$ implies,
\[
a\leq \gamma_r(x_0,r)+G(x_0,r, D\gamma(x_0,r)),
\]
for $r$ near $r_0$, for instance on $(r_1,r_0)$ for some $r_1<r_0$. Integrating the previous line thus yields,
\[
a(r_0-r_1)\leq \gamma(x_0,r_0)-\gamma(x_0,r_1)+\int_{r_1}^{r_0}G(x_0,r, D\gamma(x_0,r))\, dr.
\]
Since $(x_0,r_0)$ is a local maximum of the function given in (\ref{fn psi gamma}), we find,
\[
\gamma(x_0,r_0)-\gamma(x_0,r_1)\leq \psi_\ep(x_0,r_0) -\psi_\ep(x_0,r_1)+\int_{r_1}^{r_0}b(s)\, ds.
\]
The two previous lines thus imply,
\[
a(r_0-r_1)\leq \psi_\ep(x_0,r_0)-\psi_\ep(x_0,r_1)+\int_{r_1}^{r_0}b(r)\, dr  +\int_{r_1}^{r_0}G(x_0,r, D\gamma(x_0,r))\, dr.
\]
The definition of $\psi_\ep$ yields,
\[
a(r_0-r_1)\leq -\int_{r_1}^{r_0}H^\ep(x_0,r,D\phi(x_0))\, dr - (r_0-r_1)\alpha/2+\int_{r_1}^{r_0}b(r)\, dr  +\int_{r_1}^{r_0}G(x_0,r, D\gamma(x_0,r))\, dr.
\]
Next, we use the estimate (\ref{Hep1}) to bound the first term on the right-hand side of the previous line from above, and find,
\begin{equation}
\label{Hpsi}
a(r_0-r_1)\leq -\int_{r_1}^{r_0}H(x_0,r,D\phi(x_0))\, dr - (r_0-r_1)\alpha/4+\int_{r_1}^{r_0}b(r)\, dr  +\int_{r_1}^{r_0}G(x_0,r, D\gamma(x_0,r))\, dr.
\end{equation}

We now want to estimate the difference between $D\phi$ and $D\gamma$ in order to compare the first and last terms on the right-hand side of the previous line. To this end, we first note that, since  $(x_0,r_0)$ is a local maximum of the differentiable function (\ref{fn psi gamma}), 
\begin{align*}
D_x\gamma(x_0,r_0) &= D_x\psi_\ep(x_0,r_0) \\
&= D\phi(x_0) - \int_t^{r_0} \partial_p H^\ep( x_0, s,D\phi(x_0))D^2\phi(x_0) + \partial_x H^\ep( x_0, s,D\phi(x_0))\, ds.
\end{align*}
Rearranging the previous line and taking absolute value yields,
\begin{align*}
|D_x\gamma(x_0,r_0)- D\phi(x_0)|&\leq \int_t^{r_0} |\partial_p H^\ep( x_0, s,D\phi(x_0))D^2\phi(x_0) + \partial_x H^\ep( x_0, s,D\phi(x_0))|\, ds.
\end{align*}
The definition of $H$ (namely $H(x,t,p)=V(x,t)\cdot p$), together with the assumed bound (\ref{bdsV}) on $V$ yields  
\[
|\partial_p H^\ep(x_0, s, D\phi(x_0))|= |V(x_0,s)|\leq M.
\] 
Together with our estimate $|D_x H^\ep|\leq MpC\ep^{-1}$, as well as the estimate $||\phi||_{C^3(\rr^n)}\leq C_1$,  this yields,
\begin{align*}
|D_x\gamma(x_0,r_0)- D\phi(x_0)|&\leq \int_t^{r_0} M C_1 + M C C_1 \ep^{-1} \, ds= (r_0 - t)(M C_1 + M C C_1 \ep^{-1}).
\end{align*}
Since $r_0\leq t+\bar{h}$ we find,
\[
|D_x\gamma(x_0,r_0)- D\phi(x_0)|\leq \bar{h}(M C_1 + M C C_1 \ep^{-1}).
\]
Since $\gamma$ is continuous in $r$, we have that for $r$ near $r_0$, 
\[
|D_x\gamma(x_0,r)- D\phi(x_0)|\leq 2\bar{h}(M C_1 + M C C_1 \ep^{-1}).
\]
By our choice of $\bar{h}$ in (\ref{barh}), and since  $h\leq \bar{h}$, we have,
\[
|D_x\gamma(x_0,r_0)- D\phi(x_0)|\leq \frac{\alpha}{8M},
\]
so that,
\begin{equation*}
|H ( x,r, D_x\gamma(x_0,r))- H( x,r,D\phi(x_0))|\leq M|D_x\gamma(x_0,r)- D\phi(x_0)|\leq \frac{\alpha}{8}.
\end{equation*}

We are now ready to complete the proof of the lemma. We may assume that $r_1$ is close enough to $r_0$ so that the previous line holds. Thus, the previous line together with (\ref{Hpsi}) imply,
\begin{equation*}
a(r_0-r_1)\leq -\int_{r_1}^{r_0}H(x_0,r,D\gamma(x_0,r))\, dr - (r_0-r_1)\alpha/8+\int_{r_1}^{r_0}b(r)\, dr  +\int_{r_1}^{r_0}G(x_0,r, D\gamma(x_0,r))\, dr.
\end{equation*}
The fact that we took $G,b\in H^-( x_0,r_0, D\gamma(x_0,r_0))$ implies that, for $r_1$ close enough to $r_0$, the sum of the first, third, and fourth terms on the right-hand side of the previous line is non-positive. Thus we find,
\[
a(r_0-r_1)\leq- (r_0-r_1)\alpha/8,
\]
which yields the desired contradiction (since $r_0>r_1$) and completes the proof of the lemma.

\end{proof}

We provide the proof of item (\ref{item:superflow supersoln}); the proof of the other item is analogous. 
Throughout the remainder of this appendix we use $W$ to denote,
\[
W(x,t)=(\chi_{\Omega}(x,t)-\chi_{\bar{\Omega}^c}(x,t))_*.
\]
We split the proof into two parts, and begin by establishing that if $W$ is a supersolution then $\Omega_t$ is a superflow.
\begin{proof}[Proof of Theorem \ref{thm:soln implies flow} item (\ref{item:superflow supersoln})]
Let us suppose $W$ is a supersolution of (\ref{eq:main}) on $\rr^n\times [0,T]$. We aim to show that $(\Omega_t^{int})_{t\in [0,T]}$ is a generalized superflow with velocity $-H$. To this end, let us take $\bar{x}\in \rr^n$, $t\in (0,T)$, $r>0$, $\alpha>0$ and a smooth function $\phi:\rr^n\rightarrow \rr$ such that 
\begin{equation}
\label{phi below omega}
\{x: \phi(x)\geq 0\}\subset \Omega_t^{int}\cap B_r(\bar{x}),
\end{equation}
with $|D\phi|\neq 0$ on $\{x: \phi(x)=0\}$. By modifying  $\phi$ outside $B_{2r(x_0)}$, we may assume, without loss of generality, $||\phi||_{C^3(\rr^n)}<\infty$ and $||\phi||_{C^3(\rr^n)}$ depends on $r$ and $||\phi||_{C^3(B_{r(x_0)})}$. Let us use $C_1$ to denote,
\[
C_1=||\phi||_{C^3(\rr^n)}.
\]

Let  $\psi_\ep$ be as defined in Lemma \ref{lem:psiepsubsoln}, and fix $\ep=\ep_1$. 
 Notice that $\psi_\ep(x,t)=\phi(x)$, so according to (\ref{phi below omega}) we have $\psi(x,t)\leq W(x,t)$ for all $x\in \rr^n$. Since, according to Lemma \ref{lem:psiepsubsoln}, $\psi_\ep$ is a subsolution of (\ref{eq:main}), we can apply the comparison theorem and conclude that $\psi_\ep(x,r)\leq W(x,r)$ holds on $\rr^n\times (t,t+\bar{h})$. Due to the definition of $W$ this implies, for $h\in (0,\bar{h})$,
\begin{equation}
\label{psi pos in Omega}
\{x: \psi_\ep(x,t+h)> 0 \} \subset \Omega_{t+h}^{int}.
\end{equation}

Let us take $h\in (0,\bar{h})$ and $x\in \bar{B}_r(\bar{x})$ such that 
\[
\phi(x)-\int_t^{t+h} H( x,s, D\phi(x))\, ds -h\alpha> 0.
\]
Rearranging and then using the estimate (\ref{H^ep minusH}) yields,
\begin{align*}
\phi(x)&> \int_t^{t+h} H( x,s, D\phi(x))\, ds +h\alpha \geq \int_t^{t+h} H^\ep(x,s,D\phi(x)) - \frac{\alpha}{4}\, ds +h\alpha\\
&= \int_t^{t+h} H^\ep(x,s,D\phi(x))\, ds  +h\alpha/2= \phi(x)-\psi_\ep(x,t+h)
\end{align*}
where last equality follows from  the definition of $\psi$. So, we find $\psi_\ep(x,t+h)>0$ holds for such $x$, $h$. 

Thus, we have shown that for $h\in (0,\bar{h})$, 
\[
 \bar{B}_r(\bar{x})\cap \left\{x: \phi(x)-\int_t^{t+h} H( x,s, D\phi(x))\, ds -h\alpha> 0\right\} \subset \{x: \psi(x,t+h) > 0\}
\]
According to (\ref{psi pos in Omega}), the set on the right-hand side of the previous line is contained in $\Omega_{t+h}$. Thus we've shown,
\[
\bar{B}_r(\bar{x})\cap \left\{x: \phi(x)-\int_t^{t+h} H( x,s, D\phi(x))\, ds -h\alpha> 0\right\}\subset \Omega_{t+h}^{int},
\]
as desired. 
\end{proof}

For the other direction of item (\ref{item:superflow supersoln}), we will employ the following lemma. 
It is analogous to Lemma 2.1 of \cite{BS}. That lemma does not take time into account, but  the proof is almost identical and we omit it.
\begin{lem}
\label{analogylem21BS}
Let $\psi(x,t)$ be bounded on $\rr^n\times [0,T]$, $C^2$ in $x$ and continuous in time. Assume $(x_0, t_0)$ is such that $\psi(x_0, t_0)=0$ and $D\psi(x_0,t_0)\neq 0$. Define,
\[
\phi_k(x,t)=\psi(x,t)-k|x-x_0|^2.
\]
There exists $K>0$ and $\bar{h}>0$ such that, for all $h\in (0,\bar{h})$,
\begin{equation}
\label{pos set phi}
\{x| \phi_K(x,t_0-h)\geq 0\} \subset  B_r(x_0),
\end{equation}
and
\begin{equation}
\label{Dphi}
|D\phi_K(x,t_0-h)|\neq 0 \text{ on }\{x: \phi_K(x,t_0-h)=0\}.
\end{equation}
\end{lem}

We proceed with:
\begin{proof}[Proof of Theorem \ref{thm:soln implies flow}]
\textbf{Step 1.} Let $(\Omega_t^{int})_{t\in [0,T]}$ be a generalized superflow with velocity $V$. We aim to show that $W$ is a supersolution of (\ref{eq:main}).

To this end, let us suppose $\phi\in C^1(Q)$, $(x_0,t_0)\in Q$, and $(G,b)\in H^+(x_0,t_0,D\phi(x_0,t_0))$ are such that 
\[
(x,t)\mapsto W(x,t)+\int_0^tb(s)\, ds-\phi(x,t)
\]
has a local minimum at $(x_0,t_0)$. 
We may assume that the minimum is strict on $\{(x,t): |x-x_0|+|t-t_0|<2r\}$ for some $r$.  
In addition, by replacing $\phi(x,t)-\int_0^tb(s)\, ds$ with $\phi(x,t)-\int_0^tb(s)\, ds -\phi(x_0,t_0)+\int_0^{t_0}b(s)\, ds $, we may assume 
\begin{equation}
\label{phibx0t0}
\phi(x_0,t_0)-\int_0^{t_0}b(s)\, ds=0.
\end{equation}
 Finally, by properly modifying $\phi$ on $\{(x,t): |x-x_0|+|t-t_0|>4r\}$ we may assume that $\phi$ is bounded on $\rr^n\times [0,T]$.

If $(x_0,t_0)$ is in the interior of $\{W=1\}$ or $\{W=-1\}$, then we have $\phi_x(x_0,t_0)=0$ and $\phi_t(x_0,t_0)=0$, so that the previous line holds. Thus, from now on we assume $(x_0,t_0)\in \bdry(\{W=1\}\cup\{W=-1\})$. In particular, 
\begin{equation}
\label{x0ninOmegat}
x_0\notin \Omega_{t_0}^{int}.
\end{equation}
In addition, since  $W$ is lower-semicontinuous, we have 
\begin{equation}
\label{Wx0t0}
W(x_0,t_0)=-1.
\end{equation}

 We want to show
\[
\phi_t(x_0,t_0)+G(x_0,t_0,D\phi(x_0,t_0))\geq 0.
\]
We proceed by contradiction and assume that the previous line does not hold, so that there exists $a>0$ with 
\begin{equation}
\label{contradictAtx0t0}
\phi_t(x_0,t_0)+G(x_0,t_0,D\phi(x_0,t_0))< -a.
\end{equation}
Due to the continuity of $\phi_t$ and $G$, we have that, for $h$ small enough,
\[
\int_{t_0-h}^{t_0} \phi_t(x_0,s)\, ds + \int_{t_0-h}^{t_0}G(x_0, s,D\phi(x_0, s))\, ds <-a/2.
\]
Since we have assumed $\phi(x_0,t_0)=\int_0^{t_0} b(s)\, ds$, the previous line becomes,
 \begin{equation}
 \label{consequence of cont}
\int_0^{t_0} b(s)\, ds -\phi(x_0, t_0-h) + \int_{t_0-h}^{t_0}G(x_0, s,D\phi(x_0, s))\, ds <-a/2 .
\end{equation}
By a similar argument we find that (\ref{contradictAtx0t0}) implies, for $h$ small enough, 
 \begin{equation}
 \label{consequence of cont2}
\int_0^{t_0} b(s)\, ds -\phi(x_0, t_0-h) + \int_{t_0-h}^{t_0}G(x_0, s,D\phi(x_0, t_0))\, ds <-a/2 .
\end{equation}

\textbf{Case one.} Let us first suppose $D\phi(x_0,t_0)\neq 0$.  We define $\phi_k$ by,
\[
\phi_k(x,t)=\phi(x,t)-\int_0^tb(s)\, ds-k|x-x_0|^2.
\]
Then, according to Lemma \ref{analogylem21BS}, we have that there exist $K>0$ and $\bar{h}>0$ such that, for all $h\in (0,\bar{h})$, 
\begin{equation}
\label{use pos set phi}
\{x| \phi_K(x,t_0-h)\geq 0\} \subset  B_r(x_0),
\end{equation}
and
\begin{equation}
|D\phi_K(x,t_0-h)|\neq 0 \text{ on }\{x: \phi_K(x,t_0-h)=0\}.
\end{equation}
There exists $C$ that does not depend on $h$ with $||\phi_K(\cdot, t_0-h)||_{C^3(B_r(x_0))}\leq C$ for all $h\in (0,\bar{h})$. In addition, $(x_0, t_0)$ is a strict minimum of 
\[
(x,t)\mapsto W(x,t)-\phi_K(x,t)
\]
on $\{(x,t): |x-x_0|+|t-t_0|>4r\}$. Because $(x_0,t_0)$ is a strict minimum, we have, for any $x$, $t$ with $0<|x-x_0|+|t-t_0|<2r$,
\[
W(x,t)-\phi_K(x,t) > W(x_0,t_0)-\phi_K(x_0,t_0)=-1,
\]
where the equality follows from (\ref{phibx0t0}) and (\ref{Wx0t0}). Upon rearranging, the previous line becomes,
\[
W(x,t)> -1 +\phi_K(x,t).
\]
Let us now suppose $(x,t_0-h)$ is such that $x\in B_r(x_0)$, $0<h<r$, and $\phi_K(x,t_0-h)\geq 0$. Then, by the previous line, we have
\[
W(x,t_0-h)>-1.
\]
But, since $W$ takes only the values $1$ and $-1$, this means $W(x,t_0-h)=1$. Recalling the definition of $W$ yields $x\in \Omega_{t_0+h}^{int}$. Therefore, we've shown,
\[
\{x: \phi_K(x,t_0-h)\geq 0\}\cap B_r(x_0)\subset \Omega_{t_0-h}^{int}.
\]
But, according to (\ref{use pos set phi}), we have $\{x: \phi_K(x,t_0-h)\geq 0\}\subset B_r(x_0)$. Therefore we have, for any $h\in (0,\bar{h})$,
\[
\{x: \phi_K(x,t_0-h)\geq 0\}\subset \Omega_{t_0-h}^{int}\cap B_r(x_0).
\] 
Since $\Omega_t^{int}$ is a superflow with velocity $-H$, and since the $C^3$ norms of $\phi_K(\cdot, t_0-h)$ on $B_r(x_0)$ are uniformly bounded in $h$, we have that for $\alpha$ small enough, for any $h\in (0,\bar{h})$, and for any $\tilde{h}\in (0,h_0)$,
\[
\{x: \phi_K(x,t_0-h)-\int_{t_0-h}^{t_0-h+\tilde{h}} H(x,s,D\phi_K(x,t_0-h))\, ds-\tilde{h}\alpha \geq 0\}\cap B_r(x_0)\subset \Omega_{t_0-h+\tilde{h}}^{int}.
\] 
(Here we are thinking  of $x\mapsto \phi_k(x,t_0-h)$ as the test function, and of $t_0-h$ as the initial time.) 
Let us  now take any $h< \min\{\bar{h}, h_0\}$ and take $\tilde{h}=h$. Then, 
\[
\{x: \phi_K(x,t_0-h)-\int_{t_0-h}^{t_0} H(x,s,D\phi_K(x,t_0-h))\, ds-h\alpha \geq 0\}\cap B_r(x_0)\subset \Omega_{t_0}^{int}.
\]
By our earlier observation (\ref{x0ninOmegat}), we have $x_0\notin \Omega_{t_0}^{int}$, and hence $x_0$ is not contained in the set on the left-hand side of the previous line. Therefore,
\[
\phi_K(x_0,t_0-h)-\int_{t_0-h}^{t_0} H(x,s,D\phi_K(x_0,t_0-h))\, ds-h\alpha \leq 0.
\]
Recalling the definition of $\phi_K$ gives,
\begin{equation*}
\phi(x_0,t_0-h)-\int_0^{t_0-h}b(s)\, ds -\int_{t_0-h}^{t_0} H(x,s,D\phi_K(x_0,t_0-h))\, ds-h\alpha \leq 0.
\end{equation*}
Adding (\ref{consequence of cont}) to the previous line yields,
\[
\int_{t_0-h}^{t_0} b(s)+G(x_0, s,D\phi(x_0, s))-H(x,s,D\phi_K(x_0,t_0-h))\, ds -h\alpha <-a/2.
\]
However, the fact that this holds for all $h$ small enough contradicts that $(b,G)\in H^+(x_0, t_0, D\phi(x_0, t_0))$, completing the proof.

\textbf{Case two.} Now we consider the case  $D\phi(x_0,t_0)= 0$. 
That $(x_0,t_0)$ is a strict local minimum of $W(x,t)-\phi(x,t)+\int_0^tb(s)\, ds$ means, for $(x,t)$ near $(x_0,t_0)$,
\[
W(x,t)-\phi(x,t)+\int_0^tb(s)\, ds> W(x_0,t_0)-\phi(x_0,t_0)+\int_0^{t_0}b(s).
\]
Using (\ref{phibx0t0}) and (\ref{Wx0t0}), and rearranging, implies,
\begin{equation}
\label{W greater -1 phi}
W(x,t)> -1 + \phi(x,t)-\int_0^tb(s)\, ds
\end{equation}
for $(x,t)$ near $(x_0,t_0)$. 

Writing  the Taylor expansion of $\phi(x,t)$ at $(x_0, t_0)$  in $x$ (recalling that $D\phi(x_0,t_0)=0$)  
 yields, for $(x,t)$ near $(x_0, t_0)$ and a constant $C$ depending on $||\phi||_{C^2}$,
\begin{align*}
\phi(x,t) \geq \phi(x_0, t) - C|x-x_0|^2.
\end{align*}  
Taking $t=t_0-h$ and using the previous line to bound the second term on the right-hand side of (\ref{W greater -1 phi}) from below yields,
\[
 -1 +\phi(x_0,t_0-h)- C|x-x_0|^2-\int_{0}^{t_0-h}b(s)\, ds<W(x,t_0-h).
\]
Adding (\ref{consequence of cont2}) to the  previous line yields,
\[
-1+\int_{t_0-h}^{t_0} b(s)\, ds + \int_{t_0-h}^{t_0}G(x_0, s,D\phi(x_0, s))\, ds - C|x-x_0|^2 <-a/2 +W(x,t_0-h) .
\]
Upon rearranging we find,
\begin{equation}
\label{intbG}
 \int_{t_0-h}^{t_0}b(s)+G(x_0, s,D\phi(x_0, t_0))\, ds <-a/2 +W(x,t_0-h)+1+C|x-x_0|^2.
\end{equation}
Since $D\phi(x_0,t_0)=0$ and $H(x,t,p)=V(x,t)\cdot p$, we find $H(x_0,t_0, D\phi(x_0,t_0))=0$. The fact that $(b,G)\in H^+((x_0,t_0, D\phi(x_0,t_0))$  implies, for  $s$ near $t_0$,
\[
b(s)+G(x_0, s,D\phi(x_0, t_0))\geq H(x_0, t_0, D\phi(x_0, t_0))=0.
\]
This means that for $h$ small enough, the left-hand side of (\ref{intbG}) is nonnegative. Thus (\ref{intbG}) becomes,
\[
-a/2 +W(x,t_0-h)+1+C|x-x_0|^2>0.
\]

\textbf{Subcase (a).} Suppose there exists a sequence $(x_n,h_n)\rightarrow (x_0,0)$ with $W(x_n, t-h_n)=-1$. Then evaluating the previous line at $(x_n, t-h_n)$ and taking the limit $n\rightarrow \infty$ yields,
\[
-a/2\geq 0,
\]
which is the desired contradiction.

\textbf{Subcase (b).} Thus it is left to consider the case that there exists $r_0>0$  such that for all $h\in (0, h_0)$ and all $x\in B_{r_0}(x_0)$, we have $W(x,t)=1$. For any $h\in (0,r_0^2)$, we define the test function $\psi$ by,
\[
\psi(x)=h-|x-x_0|^2.
\]
If $x$ is such that $\psi(x)\geq 0$, then we have 
\[
|x-x_0|^2\leq h\leq r_0^2,
\]
hence $x\in B_{r_0}(x_0)$ and $W(x,t_0-h)=1$. This means,
\[
\{x|\psi(x)\geq 0\}\subset \Omega_{t_0+h}^{int}\cap B_{r_0}(x_0).
\]
We now use the fact that $\Omega$ is a superflow (we are thinking of $\psi$ as the test function and $t_0-h$ as the initial time, and take $\alpha=1/2$). We find that there exists $h_0$ (that does not depend on $h$) such that, for all $\tilde{h}\in (0,h_0)$, we have,
\[
\{x|\psi(x)-\int_{t_0-h}^{t_0-h+\tilde{h}}H(x,s,D\psi(x))\, ds - h/2>0\}\cap B_{r_0}(x_0)\subset \Omega_{t_0-h+\tilde{h}}^{int}.
\]
Let us now take $h<\min\{h_0, r_0\}$  and $\tilde{h}=h$. Then the previous line becomes,
\[
\{x|\psi(x)-\int_{t_0-h}^{t_0}H(x,s,D\psi(x)\, ds - h/2>0\}\cap B_{r_0}(x_0)\subset \Omega_{t_0}^{int}.
\]
According to (\ref{x0ninOmegat}), we have $x_0\notin \Omega_{t_0}^{int}$, and hence $x_0$ is not contained in the set on the left-hand side of the previous line. This means,
\[
\psi(x_0)-\int_{t_0-h}^{t_0}H(x,s,D\psi(x_0))\, ds - h/2\leq 0.
\]
The definition of $\psi$ yields, $\psi(x_0)=h$ and $D\psi(x_0)=0$. Since we have $H(x,t,p)=V(x,t)\cdot p$, the previous line becomes,
\[
h-h/2\leq 0,
\]
which yields the desired contradiction and completes the proof.
\end{proof}

\section*{Acknowledgements} 
The authors are very happy to acknowledge Beno\^it Perthame, who brought this problem to their attention.

\end{document}